\documentclass[10pt, reqno]{amsart}
\numberwithin{equation}{section}
\usepackage{amssymb, slashed}
\usepackage{hyperref, cite}
\usepackage{color}

\let\Re=\undefined\DeclareMathOperator*{\Re}{Re}
\let\Im=\undefined\DeclareMathOperator*{\Im}{Im}

\newcommand{\R}{\mathbb{R}}
\newcommand{\C}{\mathbb{C}}
\newcommand{\E}{\mathcal{E}}

\newcommand{\eps}{\varepsilon}

\newcommand{\N}{\mathcal{N}}

\renewcommand{\L}{\mathcal{L}}

\newtheorem{theorem}{Theorem}[section]

\newtheorem{lemma}[theorem]{Lemma}

\newtheorem{corollary}[theorem]{Corollary}

\newtheorem{proposition}[theorem]{Proposition}

\theoremstyle{definition}
\newtheorem{definition}[theorem]{Definition}

\theoremstyle{remark}

\newcommand{\qtq}[1]{\quad\text{#1}\quad}

\begin{document}

\title[NLW with inverse-square potential]{The energy-critical nonlinear wave equation with an inverse-square potential}

\author[C. Miao]{Changxing Miao}
\address{Institute for Applied Physics and Computational Mathematics, Beijing, China}
\email{miao\textunderscore changxing@iapcm.ac.cn}

\author[J. Murphy]{Jason Murphy}
\address{Missouri University of Science and Technology, Rolla, MO, USA}
\email{jason.murphy@mst.edu}

\author[J. Zheng]{Jiqiang Zheng}
\address{Institute for Applied Physics and Computational Mathematics, Beijing, China}
\email{zhengjiqiang@gmail.com}

\begin{abstract} We study the energy-critical nonlinear wave equation in the presence of an inverse-square potential in dimensions three and four.  In the defocusing case, we prove that arbitrary initial data in the energy space lead to global solutions that scatter.  In the focusing case, we prove scattering below the ground state threshold.
\end{abstract}

\maketitle

\section{Introduction}

We consider the initial-value problem for the energy-critical nonlinear wave equation (NLW) with an inverse-square potential.  The underlying linear problem is given in terms of the operator
\[
\L_a:= -\Delta + a|x|^{-2}.
\]
Here we restrict to dimensions $d\geq 3$ and values $a>-(\tfrac{d-2}{2})^2$, and we consider $\L_a$ as the Friedrichs extension of the quadratic form defined on $C_c^\infty(\R^d\backslash\{0\})$ via
\[
f\mapsto \int_{\R^d} |\nabla f(x)|^2 + a|x|^{-2}|f(x)|^2\,dx.
\]
The lower bound on $a$ guarantees positivity of $\L_a$; in fact, by the sharp Hardy inequality one finds that the standard Sobolev space $\dot H^1$ is equivalent to the Sobolev space $\dot H_a^1$ defined in terms of $\L_a$ (see Section~\ref{S:notation}).  Furthermore, when $a=0$ we recover the standard Laplacian.

We consider the following nonlinear wave equation:
\begin{equation}\label{nlw}
\begin{cases}
\partial_t^2 u + \L_a u + \mu |u|^{\frac{4}{d-2}}u = 0, \\
(u,\partial_t u)|_{t=0} = (u_0,u_1).
\end{cases}
\end{equation}
Here $u$ is a real-valued function on $\R^{1+d}$ with $d\geq 3$ and $\mu\in\{\pm1\}$ corresponds to the \emph{defocusing} and \emph{focusing} equations, respectively.  This is a Hamiltonian equation, with the conserved energy given by
\[
E_a[\vec u] = \int \tfrac12 |\nabla u|^2 + \tfrac12 |\partial_t u|^2 +  \tfrac12 a|x|^{-2}|u|^2 + \mu\tfrac{d-2}{2d}|u|^{\frac{2d}{d-2}}\,dx,
\]
where $\vec u = (u,\partial_t u)$.  The notation $E_a[f]$ should be understood as $E_a[(f,0)]$.

The operator $\L_a$ arises often in mathematics and physics in scaling limits of more complicated problems, for example in combustion theory, the Dirac equation with Coulomb potential, and the study of perturbations of space-time metrics such as Schwarzschild and Reissner--Nordstr\"om \cite{BPSTZ, KSWW, VazZua, ZhaZhe}.

One particularly interesting feature of the inverse-square potential is that it has the same scaling as the Laplacian.  In particular, one cannot in general treat $\L_a$ as a perturbation of $-\Delta$, which contributes to the mathematical interest of this particular model.  An additional consequence is that \eqref{nlw} has a scaling symmetry, namely,
\begin{equation}\label{scale}
\begin{aligned}
u(t,x)&\mapsto \lambda^{\frac{d}{2}-1} u(\lambda t,\lambda x), \\
\partial_t u(t,x) &\mapsto \lambda^{\frac{d}{2}} (\partial_t u)(\lambda t,\lambda x).
\end{aligned}
\end{equation}
This rescaling leaves the energy invariant and identifies the scaling-critical space of initial data to be the energy space $\dot H^1\times L^2$.  We therefore call \eqref{nlw} an \emph{energy-critical} equation, and indeed when $a=0$ the equation reduces to the standard energy-critical NLW, which has been the center of a great deal of research in recent years.

On the other hand, the presence of the inverse-square potential breaks translation symmetry, introducing new challenges into the analysis of \eqref{nlw}.  In particular, our work fits in the context of recent work on dispersive equations in the presence of broken symmetries, which have also attracted a great deal of interest in recent years (see e.g. \cite{Hong, IP1, IP2, IPS, Jao1, Jao2, Jao3, KKSV, KOPV, KSV, KVZ0, KVZ, PTW} and in particular \cite{KMVZZ-NLS, KMVZ, JMM, Mur} for the case of nonlinear Schr\"odinger equations with an inverse-square potential).

We consider the problem of global well-posedness and scattering for \eqref{nlw}. In the defocusing case, we will prove scattering for arbitrary data in the energy space.  In the focusing case, we will prove scattering below the ground state threshold.  These results parallel those established for the standard energy-critical NLW (see e.g. \cite{BahGer, Grillakis, Grillakis2, Kapitanski, KenigMerle, LiZha, Nakanishi, ShaStr, Struwe}), and as in many of those works we will proceed via the concentration-compactness/rigidity approach. Before comparing our work with the existing literature, however, let us state our main results more precisely.

Implicit in the statements below is the fact that any initial data in the energy space leads to a unique solution that exists at least locally in time (see Proposition~\ref{P:LWP}).  This local result leads to some restrictions on the parameter $a$, which we will state in terms of the following constant:
\begin{equation}\label{cd}
c_d=\begin{cases} \tfrac{1}{25} & d=3, \\ \tfrac{1}{9} & d=4.\end{cases}
\end{equation}
See Section~\ref{S:harmonic} and Section~\ref{S:lwp} for more details.

As above, given a solution $u$ to a (linear or nonlinear) wave equation, we write $\vec u=(u,\partial_t u)$.  We say a solution $u$ to \eqref{nlw} \emph{scatters} if there exist solutions $v_\pm(t)$ to $(\partial_t^2+\L_a)v_\pm=0$ such that
\[
\lim_{t\to\pm\infty} \| \vec u(t) - \vec v_\pm(t) \|_{\dot H^1\times L^2} =0.
\]

Our result in the defocusing case is the following theorem.
\begin{theorem}\label{T:defocusing} Let $d\in\{3,4\}$, $a>-(\frac{d-2}{2})^2+c_d$, and $\mu=+1$.  For any $(u_0,u_1)\in \dot H^1\times L^2$, the corresponding solution to \eqref{nlw} is global and scatters.
\end{theorem}

We next turn to the focusing case.  In this case, there exist global nonscattering solutions, and hence we do not expect a scattering result without some size restrictions.  Indeed, fixing $a>-(\frac{d-2}{2})^2$ and defining $\beta>0$ through the identity $a=(\tfrac{d-2}{2})^2[\beta^2-1]$, the \emph{ground state soliton} for \eqref{nlw} is the static solution to \eqref{nlw} defined by
\[
W_a(x):=[d(d-2)\beta^2]^{\frac{d-2}{4}}\bigl[\tfrac{|x|^{\beta-1}}{1+|x|^{2\beta}}\bigr]^{\frac{d-2}{2}},
\]
which arises as an optimizer of a Sobolev embedding inequality (see \cite{KMVZZ-NLS,Smets} and Section~\ref{S:variational} below).  Our result in the focusing case is a scattering result below the ground state threshold.  In the following, we write $a\wedge 0=\min\{a,0\}$ and let $\dot H_a^1$ denote the Sobolev space defined in terms of $\sqrt{\L_a}$ (see Section~\ref{S:harmonic}).

\begin{theorem}\label{T:focusing} Let $d\in\{3,4\},$ $a>-(\frac{d-2}{2})^2+c_d$, and $\mu=-1$.  Let $(u_0,u_1)\in \dot H^1\times L^2$ satisfy
\begin{equation}\label{btgs}
E_a[(u_0,u_1)] < E_{a\wedge 0}[W_{a\wedge 0}]\qtq{and} \|u_0\|_{\dot H_a^1} < \|W_{a\wedge 0}\|_{\dot H^1_{a\wedge 0}}.
\end{equation}
Then the corresponding solution to \eqref{nlw} is global and scatters.
\end{theorem}

As mentioned above, the restriction $a>-(\tfrac{d-2}{2})^2$ guarantees that the operator $\L_a$ is positive, while the further restriction $a>-(\tfrac{d-2}{2})^2+c_d$ arises in the development of the local theory for \eqref{nlw}.  Note that both results still include a range negative values of $a$, in which case the potential is attractive.  This is in contrast to many results for dispersive PDE with potentials, in which case one must consider repulsive (or perturbative) potentials in order to obtain scattering.  We restrict to dimensions $d\in\{3,4\}$ to guarantee that the nonlinearity in \eqref{nlw} is algebraic (quintic and cubic, respectively), which is primarily for technical convenience and already includes the most interesting cases.  We expect that the results extend to higher dimensions, as well.   Finally, let us point out that the results above hold without any radial assumption on the initial data.

Theorems~\ref{T:defocusing}~and~\ref{T:focusing} parallel the existing results for the standard energy-critical NLW without potential (henceforth the \emph{free NLW}).  In fact, as we will see, our results rely in an essential way on these existing results.  Comparing with the result of Kenig and Merle on the focusing NLW \cite{KenigMerle}, we find that the scattering threshold for \eqref{nlw} is the same as that for the standard NLW in the range $a>0$.  

In the free case, one also has a blowup result for solutions below the ground state energy with $\|u_0\|_{\dot H^1}>\|W_0\|_{\dot H^1}$.  The existence ground state soliton then shows that the $E_0[W_0]$ is the correct energy threshold for a simple blowup/scattering dichotomy.  The situation is completely analogous when $a\leq 0$.  On the other hand, when $a>0$ the threshold is not given in terms of the soliton $W_a$; nonetheless, the condition appearing in Theorem~\ref{T:focusing} is sharp in terms of obtaining uniform space-time bounds. In particular, the proof of Theorem~\ref{T:focusing} will show that the solutions constructed have critical space-time norms controlled by $C(E_{a\wedge 0}[W_{a\wedge 0}]-E_a[(u_0,u_1)])$ for some function $C$; we will show that this constant diverges as one approaches the threshold.  Similar results have been obtained in  \cite{KVZ, KMVZ}.

We summarize the results just mentioned in the following theorem.\begin{theorem}\label{T:blowup} Let $d\in\{3,4\}$, $a>-(\tfrac{d-2}{2})^2+c_d$, and $\mu=-1$.
\begin{itemize}
\item[(i)] If $(u_0,u_1)\in\dot H^1\times L^2$ satisfy
\[
E_a[(u_0,u_1)]<E_{a\wedge 0}[W_{a \wedge 0}] \qtq{and} \|u_0\|_{\dot H_a^1} > \|W_{a\wedge 0}\|_{\dot H_{a\wedge 0}^1},
\]
then the corresponding solution to \eqref{nlw} blows up in finite time in both time directions.
\item[(ii)] If $a>0$, then there exist a sequence of global solutions $u_n$ such that
\[
E_a[\vec u_n]\nearrow E_0[W_0]\qtq{and}\|u_n(0)\|_{\dot H_a^1}\nearrow \|W_0\|_{\dot H^1},
\]
with
\[
\lim_{n\to\infty} \|u_n\|_{L_{t,x}^{\frac{2(d+1)}{d-2}}(\R\times\R^d)} = \infty. 
\]
\end{itemize}
\end{theorem}

Our main focus in this paper is on the scattering results, Theorem~\ref{T:defocusing}~and~\ref{T:focusing}.  In the remainder of the introduction, we will primarily discuss the proof of these results.  The proof of Theorem~\ref{T:blowup} is relatively straightforward and will be explained in Section~\ref{S:blowup}.

The strategy of proof for Theorems~\ref{T:defocusing}~and~\ref{T:focusing} is the concentration-compactness approach to induction on energy, often referred to as the Kenig--Merle roadmap.  In particular, supposing that either theorem is false, we show that there would exist a special type of solution to \eqref{nlw} (constructed as a minimal blowup solution) possessing certain compactness properties. We then preclude the possibility that such solutions can exist. The precise notion of compactness is given by the following:

\begin{definition}[Almost periodic solution]\label{D:AP} Let $\vec u$ be a nonzero solution to \eqref{nlw} on an interval $I$. We call $\vec u$ \emph{almost periodic} (modulo symmetries) if there exist a spatial center $x:I\to\R^d$ and frequency scale $N:I\to(0,\infty)$ such that the set
\[
\bigl\{\bigl(N(t)^{\frac{d}{2}-1}u(t,N(t)[x-x(t)]), N(t)^{\frac{d}{2}}(\partial_t u)(t,N(t)[x-x(t)])\bigr):t\in I\bigr\}
\]
is pre-compact in $\dot H^1\times L^2$.
\end{definition}

The first main step of the proof, the reduction to almost periodic solutions, appears as Theorem~\ref{T:reduction} below.  The general strategy is well-established, with essentially two key ingredients: (i) a linear profile decomposition adapted to a Strichartz estimate (see Proposition~\ref{P:LPD}) and (ii) a corresponding `nonlinear profile decomposition'.  Both of these steps involve additional difficulties in the setting of equations with broken symmetries.  For (i), the new difficulty is related mostly to understanding the convergence of certain linear operators that arise due to the failure of translation symmetry (see Lemma~\ref{L:COO} and Lemma~\ref{L:COO2}).  For (ii), the key difficulty arises from the construction of (scattering) nonlinear solutions associated to profiles with a translation parameter tending to infinity.  Because of the broken translation symmetry, one cannot simply solve \eqref{nlw} and then incorporate the translation.  Instead, roughly speaking, one constructs a solution to the \emph{free} NLW, incorporates the translation, and (because the profile lives far from the origin) shows that the result is an approximate solution to \eqref{nlw}.  An application of the stability theory for \eqref{nlw} then yields the true solution, as desired.  The construction of a suitable (scattering) solution to the free NLW relies on the full strength of \cite{KenigMerle, BahGer}. For more details, see Proposition~\ref{P:NP}.

With the necessary ingredients in place, the reduction to almost periodic solutions (Theorem~\ref{T:reduction}) follows along fairly standard lines (see Section~\ref{S:reduction}).  After this, it is useful to make some further reductions to the class of almost periodic solutions that we consider.  Recall from Definition~\ref{D:AP} above that almost periodic solutions are described in terms of a spatial center $x(t)$ and frequency scale $N(t)$.  Because of the result in Proposition~\ref{P:NP}, we firstly find that we must have $x(t)\equiv 0$; indeed, as described above, profiles with translation parameters tending to infinity correspond to scattering solutions and hence do not arise in the construction of minimal blowup solutions.  In Theorem~\ref{T:refine}, we adapt arguments of \cite{KenigMerle} to further reduce to two scenarios, namely, the `forward global' scenario and the `self-similar' scenario.

The preclusion of both scenarios relies heavily on certain virial/Morawetz identities.  For NLW with a general potential $V(x)$, these identities will involve a term of the form $-\tfrac12 x\cdot\nabla V$ (see Lemma~\ref{L:virial}, for example).  For repulsive potentials (i.e. those satisfying $x\cdot \nabla V\leq 0$), this term generally has a good sign and is amenable to deriving monotonicity formulas; this explains in part why scattering results are often restricted to the case of repulsive potentials.  For the inverse square potential $V(x)=a|x|^{-2}$, one has the identity $\tfrac12 x\cdot\nabla V = -V$ (a consequence of scaling).  In particular, while the potential is repulsive only for $a\geq 0$, this identity ultimately allows us to prove suitable virial/Morawetz estimates even when $a<0$.

To deal with the forward-global scenario, we use a localized virial estimate (see Section~\ref{S:FG}).  The implementation of this argument is fairly straightforward due to the fact that we have $x(t)\equiv 0$.  This should be contrasted with \cite{KenigMerle}, where the authors must argue (using Lorentz boosts) that minimal blowup solutions have zero momentum, which then allows for sufficient control over $x(t)$ to run the localized virial argument.  In the focusing case, we must also rely on the coercivity given by sharp Sobolev embedding and the fact that the minimal blowup solution is below the ground state in the sense of \eqref{btgs} (see Lemma~\ref{L:coercive}).

Finally, we rule out the self-similar scenario in Section~\ref{S:SS}.  In this scenario, the solution $u$ blows up at time $t=0$ and is supported at each $t>0$ in $B_t(0)$, with $N(t)=t^{-1}$. To rule out such solutions, we firstly prove a virial/Morawetz estimate (Proposition~\ref{SSvirial}), which in particular implies the vanishing of the quantity
\[
\partial_t u + x\cdot\nabla u + \tfrac{d-2}{2}\tfrac{u}{t}
\]
along a sequence $t_n\to 0$.  This particular estimate is similar to one used to study wave maps (appearing e.g. in \cite{GrWM, Tao, ST}), and is particularly closely related to the estimate appearing in \cite{DJKM}.  Using Proposition~\ref{SSvirial} and almost periodicity, we are able to extract a true self-similar solution to \eqref{nlw}, that is, a (nonzero) solution of the form
\[
v(t,x)=(1+t)^{-[\frac{d}{2}-1]}f(\tfrac{x}{1+t}).
\]
It follows that $f$ solves a degenerate elliptic PDE in the unit ball and vanishes near the boundary in a certain sense (see Proposition~\ref{P:degenerate}).  At this point, we find ourselves essentially in the same position as Kenig and Merle \cite{KenigMerle}; indeed, for this portion of the argument the difficulty arises only at the boundary of the unit ball, and we may safely include the potential term in the nonlinearity.  Following closely the arguments of \cite{KenigMerle}, we employ a change of variables to remove the degeneracy and invoke unique continuation results to deduce that $f\equiv 0$, yielding a contradiction and completing the proof of Theorems~\ref{T:defocusing} and \ref{T:focusing}.

As just described, our arguments in the self-similar scenario differ from Kenig and Merle \cite{KenigMerle} essentially only in the extraction of the elliptic solution.  In particular, \cite{KenigMerle} employs a self-similar change of variables, while we utilize a virial/Morawetz estimate (actually closer to the spirit of the arguments of \cite{DJKM}).  In fact, one observes that setting $a= 0$ throughout Section~\ref{S:SS} leads to a modified proof of the preclusion of self-similar almost periodic solutions for the free NLW.

The rest of the paper is organized as follows:
\begin{itemize}
\item In Section~\ref{S:notation} we first set up some notation and record some useful lemmas, including the useful virial/Morawetz identity.  In Section~\ref{S:harmonic}, we collect some harmonic analysis tools adapted to $\L_a$, including some results concerning the equivalence of Sobolev spaces as well as Strichartz esimates. In Section~\ref{S:lwp} we record the basic local well-posedness and stability theory for \eqref{nlw}. Finally, in Section~\ref{S:variational}, we record some results related to the sharp Sobolev embedding and the ground state solution for \eqref{nlw}.
\item In Section~\ref{S:CC} we develop concentration compactness tools for \eqref{nlw}.  The main result of this section is the linear profile decomposition, Proposition~\ref{P:LPD}.
\item In Section~\ref{S:exist}, we prove the existence of minimal blowup solutions under the assumption that Theorem~\ref{T:defocusing} or Theorem~\ref{T:focusing} fails (cf. Theorem~\ref{T:reduction}).  We then refine the class of solutions that we need to consider (see Theorem~\ref{T:refine}).
\item In Section~\ref{S:FG} we preclude the forward global case of Theorem~\ref{T:refine}.
\item In Section~\ref{S:SS} we preclude the self-similar case of Theorem~\ref{T:refine}, thus completing the proof of Theorems~\ref{T:defocusing}~and~\ref{T:focusing}.
\item Finally, in Section~\ref{S:blowup}, we give the proof of Theorem~\ref{T:blowup}.  
\end{itemize}

\subsection*{Acknowledgements} C.~M. and J.~Z. were supported by NSFC Grants 11831004.  Part of this work was completed while J.~M. was supported by the NSF postdoctoral fellowship DMS-1400706 at UC Berkeley.  We are grateful to R. Killip for some useful discussions related to Lemma~\ref{L:COO2} below.  We are also grateful to H. Jia for helping us to understand the works \cite{DJKM, KenigMerle}, which played a key role in developing Section~\ref{S:SS} of this paper.

\section{Notation and lemmas}\label{S:notation}

We will employ some standard geometric notation.  We let 
\[
g_{\alpha\beta}=\text{diag}(-1,1,\dots,1)
\] denote the standard Minkowski metric on $\R^{1+d}$ and denote the inverse metric by $g^{\alpha\beta}$.  Greek indices take values in $\{0,1,\dots,d\}$ while Roman indices take values in $\{1,\dots,d\}$. We employ Einstein summation convention, and we raise and lower indices with respect to the metric: $\partial^\alpha=g^{\alpha\beta}\partial_\beta$.  For example, \eqref{nlw} may be written
\[
\partial^\alpha\partial_\alpha u = a|x|^{-2}u + \mu |u|^{\frac{4}{d-2}}u.
\]
We write $(x^{\alpha})$ for space-time coordinates, with $x^0=t$. We use $\nabla$ for the gradient in the spatial variables only; the space-time gradient will be denoted by $\nabla_{t,x}$.

We use the standard Lebesgue spaces $L^p(\R^d)$, as well as the mixed space-time norms $L_t^q L_x^r(\R^{1+d})$, defined by
\[
\|u\|_{L_t^q L_x^r(\R^{1+d})} = \bigl\| \,\|u(t)\|_{L_x^r(\R^d)}\, \|_{L_t^q(\R)}.
\]
We write $q'\in[1,\infty]$ for the dual exponent of $q\in[1,\infty]$, i.e. the solution to $\tfrac{1}{q}+\tfrac{1}{q'}=1$.

We write $A\lesssim B$ to denote $A\leq CB$ for some $C>0$.  We can similarly define $A\gtrsim B$.  We write $a\pm$ to denote a quantity of the form $a\pm\eps$ for some small $\eps>0$.

We introduce a mapping $T_a$ that takes a pair of real-valued functions $(\phi,\psi)$ and returns a single complex-valued function defined by
\begin{equation}\label{T}
T_a(\phi,\psi):=\phi+i\L_a^{-\frac12}\psi.
\end{equation}
We apply this mapping to $\vec u = (u,\partial_t u)$, where $u$ solves \eqref{nlw}, that is,
\begin{equation}\label{U}
T_a\vec u= u + i\L_a^{-\frac12}\partial_t u.
\end{equation}
Thus $u=\Re T_a\vec u$, and $u$ solves \eqref{nlw} with data in $\dot H_a^1\times L^2$ if and only if the complex-valued function $v:=T_a\vec u$ solves
\begin{equation}\label{Tnlw}
i\partial_t v - \L_a^{\frac12}v - \mu \L_a^{-\frac12}(|\Re v|^{\frac{4}{d-2}}\Re v)=0
\end{equation}
with data in $\dot H_a^1$.  In these variables we have
\begin{equation}\label{Etilde}
E_a[\vec u] = \tilde E_a[T_a\vec u]:= \int \tfrac12 |\L_a^{\frac12}T_a\vec u|^2 + \mu \tfrac{d-2}{2d}|\Re T_a\vec u|^{\frac{2d}{d-2}}\,dx.
\end{equation}

We will need the following refinement of Fatou's lemma in Section~\ref{S:CC}.

\begin{lemma}[Refined Fatou, \cite{BrezisLieb}]\label{L:fatou} Let $1\leq p<\infty$ and let $\{f_n\}$ be bounded in $L^p$.  If $f_n\to f$ almost everywhere, then
\[
\lim_{n\to\infty} \int \bigl| |f_n|^p - |f_n-f|^p - |f|^p \bigr|\,dx =0.
\]
\end{lemma}

Finally, we record the following virial identity that will be used on a few occasions below.   The proof follows from direct computation and integration by parts.
\begin{lemma}[Virial identity]\label{L:virial} Fix a weight $w:\R^d\to\R$ and a solution $u:\R^{1+d}\to\R$ to the wave equation
\[
\partial_\alpha \partial^\alpha u = Vu+G'(u).
\]
Then
\begin{align*}
\partial_t \int -\partial_tu[\nabla u\cdot\nabla w + \tfrac12 u\Delta w]\,dx & = \int \nabla u \cdot \nabla^2 w\nabla u + \Delta w[\tfrac12 uG'(u) - G(u)] \\
& \quad\quad  -\tfrac12 \nabla w\cdot \nabla V u^2 - \tfrac14 \Delta\Delta w(u^2)\,dx.
\end{align*}
\end{lemma}

\subsection{Harmonic analysis tools}\label{S:harmonic}  A harmonic analysis toolkit adapted to $\L_a$ was developed in \cite{KMVZZ-Sobolev}.  In this section, we will import several relevant results.  We will also record some Strichartz estimates adapted to the linear wave equation with inverse-square potential, which were established in \cite{BPSTZ}.

For $r\in(1,\infty)$ we write $\dot H_a^{1,r}$ and $H_a^{1,r}$ for the homogeneous and inhomogeneous Sobolev spaces defined in terms of $\L_a$; these have norms
\[
\|f\|_{\dot H_a^{1,r}} = \|\sqrt{\L_a} f\|_{L^r},\quad \|f\|_{H_a^{1,r}} = \|\sqrt{1+\L_a} f\|_{L^r}.
\]
When $r=2$ we write $\dot H_a^{1,2}=\dot H_a^{1}$.

Let us introduce the parameter
\[
\sigma = \tfrac{d-2}{2}-\bigl[(\tfrac{d-2}{2})^2+a\bigr]^{\frac12}.
\]
One of the main results in \cite{KMVZZ-Sobolev} is the following result concerning the equivalence of Sobolev spaces.

\begin{lemma}[Equivalence of Sobolev spaces, \cite{KMVZZ-Sobolev}]\label{L:Sobolev} Let $d\geq 3$, $a>-(\tfrac{d-2}{2})^2$, and $s\in (0,2)$.
\begin{itemize}
\item If $p\in (1,\infty)$ satisfies $\tfrac{s+\sigma}{d}<\tfrac{1}{p}<\min\{1,\tfrac{d-\sigma}{d}\},$ then
\[
\| |\nabla|^s f\|_{L^p} \lesssim \| \L_a^{\frac{s}{2}}f\|_{L^p}.
\]
\item If $p\in(1,\infty)$ satisfies $\max\{\tfrac{s}{d},\tfrac{\sigma}{d}\}<\tfrac{1}{p}<\min\{1,\tfrac{d-\sigma}{d}\},$ then
\[
\| \L_a^{\frac{s}{2}}f\|_{L^p} \lesssim \| |\nabla|^s f\|_{L^p}.
\]
\end{itemize}
\end{lemma}

We will use Littlewood--Paley projections defined through the heat kernel, i.e.
\[
P_N^a = e^{-\L_a/N^2}-e^{-4\L_a/N^2},
\]
where $N\in2^{\mathbb{Z}}$.  As was shown in \cite{LS,MS}, the heat kernel $e^{-t\L_a}(x,y)$ has upper and lower bounds of the form
\begin{equation}\label{heat}
C(1\wedge \tfrac{\sqrt{t}}{|x|})^{\sigma}(1\wedge \tfrac{\sqrt{t}}{|y|})^\sigma e^{-|x-y|^2/ct}.
\end{equation}

To state results, it will be useful to define the exponent
\[
q_0=\begin{cases} \infty & \text{if }a\geq 0, \\ \tfrac{d}{\sigma} & \text{if} -(\tfrac{d-2}{2})^2<a<0.\end{cases}
\]
We write $q_0'$ for the dual exponent in both cases.  We record the harmonic analysis tools we need in the following proposition.
\begin{proposition}[Harmonic analysis tools, \cite{KMVZZ-Sobolev}]\label{P:HAT} Let $q_0'<q\leq r<q_0$.
\begin{itemize}
\item We have the following expansion:
\[
f=\sum_{N\in 2^{\mathbb{Z}}} P_N^a f \qtq{as elements of}L^r.
\]
\item We have the following Bernstein estimates:
\begin{itemize}
\item The operators $P_N^a$ are bounded on $L^r$.
\item The operators $P_N^a$ are bounded from $L^q$ to $L^r$ with norm bounded by $N^{\frac{d}{q}-\frac{d}{r}}$.
\item For any $s\in\R$,
\[
N^s\|P_N^a f\|_{L^r} \sim \| \L_a^{\frac{s}{2}} P_N^a f\|_{L^r}.
\]
\end{itemize}
\item We have the square function estimate:
\[
\biggl\| \biggl( \sum_{N\in 2^{\mathbb{Z}}} |P_N^a f|^2\biggr)^{\frac12}\biggr\|_{L^r} \sim \| f\|_{L^r}.
\]
\end{itemize}
\end{proposition}

We next turn to Strichartz estimates for the linear wave equation with inverse-square potential.  Here we import results of Burq, Planchon, Stalker, and Tahvildar-Zadeh \cite{BPSTZ}, specifically Theorem~5 and Theorem~9 therein.  We state the estimates in terms of the operators $e^{\pm it\sqrt{\L_a}}$ and specialize to dimensions $d\in\{3,4\}$.

\begin{proposition}[Strichartz]\label{P:Strichartz} Let $q,r\geq 2$ satisfy the wave admissibility condition
\[
\tfrac{1}{q}+\tfrac{d-1}{2r}\leq \tfrac{d-1}{4},
\]
where in $d=3$ we additionally require $q,\tilde q>2$.  Define $\gamma$ via the scaling relation
\[
\tfrac{1}{q}+\tfrac{d}{r}=\tfrac{d}{2}-\gamma.
\]

For any time interval $I$ we have
\begin{align*}
\| e^{\pm it\sqrt{\L_a}} f\|_{L_t^q L_x^r(I\times\R^d)} \lesssim \|f\|_{\dot H^{\gamma}(\R^d)}
\end{align*}
provided the following conditions hold:
\begin{itemize}
\item If $d=3$, then we require
\[
-\min\{1,\sqrt{a+\tfrac94}-\tfrac12,\sqrt{a+\tfrac14}+1\}<\gamma<\min\{2,\sqrt{a+\tfrac94}+\tfrac12,\sqrt{a+\tfrac14}+1-\tfrac1q\}.
\]
\item If $d=4$, then we require
\[
-\min\{\tfrac56,\sqrt{a+4}-\tfrac76,\sqrt{a+1}+1\}<\gamma<\min\{\tfrac52,\sqrt{a+4}+\tfrac12,\sqrt{a+1}+1-\tfrac1q\}.
\]
\end{itemize}
\end{proposition}

We will also need an inhomogeneous estimate.  In particular, using Proposition~\ref{P:Strichartz}, Lemma~\ref{L:Sobolev}, and the Christ--Kiselev lemma, we have the following:
\begin{corollary}[Strichartz] Let $q,r,\gamma$ be as in Proposition~\ref{P:Strichartz} and let $\tilde q,\tilde r,\tilde\gamma$ be defined similarly.  Suppose $q,\tilde q>2$.  Then for any time interval $I\ni t_0$, we have
\[
\biggl\| \int_{t_0}^t e^{i(t-s)\sqrt{\L_a}}F(s)\,ds\biggr\|_{L_t^q L_x^r(I\times\R^d)} \lesssim \| |\nabla|^{\gamma+\tilde\gamma} F\|_{L_t^{\tilde q'} L_x^{\tilde r'}(I\times\R^d)}.
\]
\end{corollary}

\subsection{Local well-posedness and stability}\label{S:lwp}

We next develop the local theory for \eqref{nlw}, including a stability result.  As the arguments are rather standard, we will be rather brief.  As usual, the results rely primarily on Strichartz estimates, which were recorded in the previous section.  We will construct solutions that lie locally in $L_t^\infty(\dot H^1\times L^2)$ as well as the Strichartz space
\begin{equation}\label{Snorm}
S(I):=L_{t,x}^{\frac{2(d+1)}{d-2}}(I\times\R^d) \cap L_t^{\frac{d+2}{d-2}}L_x^{2(\frac{d+2}{d-2})}(I\times\R^d).
\end{equation}
Note that the scaling of $S$ corresponds to $\gamma = 1$ in the Strichartz estimates appearing above, and that we are able to use these spaces provided we choose
\[
a>\begin{cases} -\tfrac14 + \tfrac1{25} & d=3 \\ -1 + \tfrac19 & d=4.\end{cases}
\]
This is the origin of the constant $c_d$ defined in \eqref{cd} and appearing in the statements of the main results, Theorem~\ref{T:defocusing} and Theorem~\ref{T:focusing}.

Writing the Duhamel formulation of \eqref{nlw}, namely,
\[
u(t)=\cos(t\sqrt{\L_a})u_0+\frac{\sin(t\sqrt{\L_a})}{\sqrt{\L_a}}u_1 -\mu\int_0^t \frac{\sin((t-s)\sqrt{\L_a})}{\sqrt{\L_a}}(|u|^{\frac{4}{d-2}}u)(s)\,ds,
\]
we can run a contraction mapping in the space $L_t^\infty(\dot H^1\times L^2)\cap S$ by relying on the nonlinear estimate
\[
\biggl\| \int_0^t  \frac{\sin((t-s)\sqrt{\L_a})}{\sqrt{\L_a}}(|u|^{\frac{4}{d-2}}u)(s)\,ds\biggr\|_{L_t^\infty \dot H^1\cap S} \lesssim \| |u|^{\frac{4}{d-2}}u\|_{L_t^1 L_x^2} \lesssim \|u\|_{S}^{\frac{d+2}{d-2}}.
\]
The conclusion is the following local result.
\begin{proposition}[Local well-posedness]\label{P:LWP} Let $(u_0,u_1)\in \dot H^1\times L^2$, $d\in\{3,4\}$, and $a>-(\frac{d-2}{2})^2+c_d$.

There exists $\eta_0$ such that if
\[
\|\cos(t\sqrt{\L_a})u_0+\frac{\sin(t\sqrt{\L_a})}{\sqrt{\L_a}}u_1\|_{S(I)}< \eta
\]
for some $0<\eta<\eta_0$, then there exists a solution to \eqref{nlw} on $I$ satisfying $\|u\|_{S(I)}\lesssim \eta$.

In particular, data in $\dot H^1\times L^2$ leads to local-in-time solutions.

The solution may be extended as long as the $S$-norm remains finite, and if the solution is global with $\|u\|_{S(\R)}<\infty$ then the solution scatters in both time directions.

Finally, given a final state $(u_0^+,u_1^+)\in \dot H^1\times L^2$, we may construct a solution on an interval $(T,\infty)$ that scatters to $(u_0^+,u_1^+)$ as $t\to\infty$.  A similar result holds backward in time.
\end{proposition}

Standard arguments relying primarily on the same Strichartz estimates as above yield the following stability result for \eqref{nlw}.

\begin{proposition}[Stability]\label{P:stab} Let $d,a$ be as in Proposition~\ref{P:LWP}.  $I$ be a time interval and let $\tilde u$ satisfy
\[
\partial_t^2\tilde u + \L_a\tilde u + \mu |\tilde u|^{\frac{4}{d-2}}\tilde u + e=0
\]
for some function $e:I\times\R^d\to\R$.  Suppose that
\[
\|\tilde u\|_{S(I)}+\|(\tilde u(t_0),\partial_t \tilde u(t_0))\|_{\dot H^1\times L^2} \leq L
\]
for some $t_0\in I$ and $L>0$.  There exists $\eps_0=\eps_0(L)$ such that for $0<\eps<\eps_0$ we have the following: if
\[
\|\vec u_0-(\tilde u(t_0),\partial_t \tilde u(t_0))\|_{\dot H^1\times L^2} + \|e\|_{L_t^1 L_x^2(I\times\R^d)} <\eps,
\]
then there exists a solution $u$ to \eqref{nlw} on $I$ with $\vec u(t_0)=\vec u_0$ satisfying
\[
\| u-\tilde u\|_{S(I)} \lesssim_L \eps \qtq{and} \|\vec u\|_{L_t^\infty(I;\dot H^1\times L^2)}+\|u\|_{S(I)}\lesssim 1.
\]
\end{proposition}

For an introduction to these types of results, we refer the reader to \cite{KVClay}.

We remark that by applying the transformation $T_a$ introduced above, one has equivalent local well-posedness and stability results stated in terms of the equation \eqref{Tnlw} with initial data in $\dot H^1$.  We will use both versions of these results below.

\subsection{Variational analysis}\label{S:variational}
In this section we record results related to the sharp Sobolev embedding
\begin{equation}\label{sharp-sobolev}
\|f\|_{L^{\frac{2d}{d-2}}(\R^d)} \leq C_a \|f\|_{\dot H_a^1(\R^d)},
\end{equation}
where $C_a$ denotes the sharp constant.

Much of the analysis that we need was carried out in \cite{KMVZZ-NLS}; see also \cite{Smets}.

For $a>-(\frac{d-2}{2})^2$, we define $\beta>0$ by $a=(\tfrac{d-2}{2})^2[\beta^2-1]$.  We may also write $\sigma=\tfrac{d-2}{2}(1-\beta)$.  The ground state soliton is defined by
\[
W_a(x) = [d(d-2)\beta^2]^{\frac{d-2}{4}}\bigl[ \tfrac{|x|^{\beta-1}}{1+|x|^{2\beta}}\bigr]^{\frac{d-2}{2}}.
\]
We have that $W_a$ solves
\[
\L_aW_a - |W_a|^{\frac{4}{d-2}}W_a = 0,
\]
and
\[
\|W_a\|_{\dot H_a^1}^2 = \|W_a\|_{L^{\frac{2d}{d-2}}}^{\frac{2d}{d-2}}=\tfrac{\pi d(d-2)}{4}\bigl[\tfrac{2\sqrt{\pi}\beta^{d-1}}{\Gamma(\frac{d+1}{2})}\bigr]^{\frac{2}{d}}.
\]
Note that the first identity above and \cite[Proposition~7.2]{KMVZZ-NLS} imply
\begin{equation}\label{Ca-id}
C_a=\|W_{a\wedge 0}\|_{\dot H_{a\wedge 0}^1}^{-\frac{2}{d}}.
\end{equation}

In Section~\ref{S:embedding}, we will need to construct scattering nonlinear solutions to \eqref{nlw} that are parametrized by a spatial center approaching infinity.  To do this, we need to approximate by solutions to the nonlinear wave equation without potential; in particular, we need to rely on the scattering result of \cite{KenigMerle}.  Consequently, in the focusing case we need to be sure that initial data lying below the threshold stated in Theorem~\ref{T:focusing} also lie below the appropriate threshold for the equation without potential.  This fact is guaranteed by the following corollary.

\begin{corollary}[Comparison of thresholds]{\label{C:comparison}} Let $a>-(\tfrac{d-2}{2})^2$.  Then
\[
E_{a\wedge 0}[W_{a\wedge 0}] \leq E_{0}[W_{0}]\qtq{and} \|W_{a\wedge 0}\|_{\dot H_{a\wedge 0}^1} \leq \|W_{0}\|_{\dot H_{0}^1}.
\]
\end{corollary}

\begin{proof} There is nothing to prove when $a\geq 0$, so let us fix $a<0$.  We begin by observing that
\[
\|f\|_{\dot H_a^1}<\|f\|_{\dot H_0^1},
\]
which implies $C_0\leq C_a$; indeed,
\[
\|f\|_{L^{\frac{2d}{d-2}}} \leq C_a \|f\|_{\dot H_a^1} <C_a \|f\|_{\dot H_0^1}.
\]
In light of \eqref{Ca-id}, we have
\[
\|W_a\|_{\dot H_a^1} \leq \|W_0\|_{\dot H_0^1},
\]
which is one of the desired inequalities.  For the remaining inequality, we again call on \eqref{Ca-id} and use the inequality just established to observe that
\[
E_a(W_a)=\tfrac{1}{d}\|W_a\|_{\dot H_a^1}^2 \leq \tfrac{1}{d}\|W_0\|_{\dot H_0^1}^2 = E_0(W_0).
\]
This completes the proof.  \end{proof}

Finally, we record the following lemma, which is almost identical to \cite[Corollary~7.6]{KMVZZ-NLS} and in particular follows from the same proof appearing there.

\begin{lemma}[Coercivity]\label{L:coercive} Let $d\geq 3$ and $a>-(\frac{d-2}{2})^2$.  Suppose $u:I\times\R^d\to\R$ is a solution to \eqref{nlw} with $\mu=-1$ and initial data $\vec u_0\in \dot H^1\times L^2$ satisfying
\[
E_a[\vec u_0]\leq(1-\delta)E_{a\wedge 0}[W_{a\wedge 0}]
\]
for some $\delta>0$.  If $\|u_0\|_{\dot H_a^1} \leq \|W_{a\wedge 0}\|_{\dot H^1_{a\wedge 0}}$, then for all $t\in I$:
\begin{itemize}
\item $\|u(t)\|_{\dot H_a^1}\leq (1-\delta')\|W_{a\wedge 0}\|_{\dot H^1_{a\wedge 0}}$

\smallskip

\item $\int |\L_a u(t,x)|^2 - |u(t,x)|^{\frac{2d}{d-2}}\,dx \gtrsim_\delta \|u(t)\|_{\dot H_a^1}^2$,

\smallskip

\item $E_a(\vec{u})\sim_\delta \|\vec u(t)\|_{\dot H^1\times L^2}^2$,
\end{itemize}
for some $\delta'>0$ depending on $\delta$.
\end{lemma}

\section{Concentration compactness}\label{S:CC}

A key step in the proofs of Theorem~\ref{T:defocusing} and Theorem~\ref{T:focusing} is to prove that if the result is false, then we can construct minimal blowup solutions with good compactness properties. This will be carried out in Section~\ref{S:exist} (see Theorem~\ref{T:reduction} and Theorem~\ref{T:refine} below).  In the present section, we will develop a key technical ingredient needed for this step, namely, a linear profile decomposition adapted to Strichartz estimates for $e^{it\sqrt{\L_a}}$ (see Proposition~\ref{P:LPD}).

We introduce the following notation, which helps keep track of the lack of translation symmetry in $\L_a$.

\begin{definition}\label{D:Ln}
Given a sequence $\{y_n\}\subset\R^d$, we define
\begin{equation*}
\L_a^n:=-\Delta+ \tfrac{a}{|x+y_n|^2} \quad\text{and}\quad \L_a^\infty:=\begin{cases} -\Delta+\frac{a}{|x+y_\infty|^2} \quad &\text{if}\quad y_n\to y_\infty\in\R^d,\\
-\Delta&\text{if}\quad |y_n|\to\infty.
\end{cases}
\end{equation*}
\end{definition}
We therefore have $\L_a[\phi(\cdot-y_n)]=[\L_a^n\phi](\cdot-y_n)]$.

\begin{proposition}[Linear profile decomposition]\label{P:LPD}  Let $f_n$ be a bounded sequence in $\dot H_a^1$.  Passing to a subsequence, there exist $J^*\in\{0,1,\dots,\infty\}$, profiles $\{\phi^j\}_{j=1}^{J^*}\subset \dot H_a^1$, scales $\{\lambda_n^j\}_{j=1}^{J^*}\subset (0,\infty)$, and space-time positions $\{(t_n^j,x_n^j)\}\subset\R^{1+d}$ such that for any finite $0\leq J\leq J^*$ we have the decomposition
\begin{equation}\label{E:LP0}
f_n=\sum_{j=1}^J \phi_n^j + r_n^J,\qtq{where} \phi_n^j(x) =(\lambda_n^j)^{-(\frac{d}{2}-1)}\bigl[e^{it_n^j\sqrt{\L_a^{n_j}}}\phi^j\bigr](\tfrac{x-x_n^j}{\lambda_n^j}),
\end{equation}
where $\L_a^{n_j}$ is as in Definition~\ref{D:Ln} with $y_n^j=\frac{x_n^j}{\lambda_n^j}$.  This decomposition satisfies the following properties for any finite $1\leq J\leq J^*$:
\begin{itemize}
\item The remainder term satisfies
\begin{align}\label{E:LP1}
\lim_{J\to J^*} \limsup_{n\to\infty} \|e^{it\sqrt{\L_a}} r_n^J\|_{L_{t,x}^{\frac{2(d+1)}{d-2}}
(\R\times\R^d)}&=0, \\
\lim_{n\to\infty} (\lambda_n^J)^{\frac{d}{2}-1}[e^{-it_n^J \sqrt{\L_a}}r_n^J](\lambda_n^J x + x_n^J)&\rightharpoonup 0 \qtq{weakly in}\dot H^1. \label{weaklpd}
\end{align}
\item For $j\neq k$ we have the orthogonality condition
\begin{align}\label{E:LP5}
\lim_{n\to\infty} \big|\log\tfrac{\lambda_n^j}{\lambda_n^k} \big|+\tfrac{|x_n^j-x_n^k|^2}{\lambda_n^j\lambda_n^k} +\tfrac{|t_n^j-t_n^k|^2}{\lambda_n^j\lambda_n^k}=\infty.
\end{align}
\item We also have the decouplings for each finite $0\leq J\leq J^*$:
\begin{align}
 \lim_{n\to\infty}\Bigl\{\|f_n\|_{\dot H^1_a}^2-\sum_{j=1}^J\|\phi_n^j\|_{\dot H_a^1}^2 -\|r_n^J\|_{\dot H^1_a}^2\Bigr\}=0,\label{E:LP2}\\
\lim_{n\to\infty}\big\{\|f_n\|_{L_x^{\frac{2d}{d-2}}}^{\frac{2d}{d-2}}-\sum_{j=1}^J\|\phi_n^j\|_{L_x^{\frac{2d}{d-2}}}^{\frac{2d}{d-2}}-\|r_n^J\|_{L_x^{\frac{2d}{d-2}}}^{\frac{2d}{d-2}}\big\}=0.\label{E:LP6}
\end{align}
\end{itemize}
Finally, we may assume that for each $j$ either $t_n^j\equiv 0$ or $t_n^j\to \pm \infty$ and either $x_n^j\equiv 0$ or $(\lambda_n^j)^{-1}|x_n^j|\to\infty$.
\end{proposition}

The strategy for proving Proposition~\ref{P:LPD} is well-established: we remove one bubble at a time until the Strichartz norm is depleted.  The key to isolating an individual bubble is to first identify a scale for concentration, which can be done via a refinement of the usual Strichartz estimate.  One then finds a space-time position for concentration via H\"older's inequality.  The broken space translation symmetry introduces some additional technical difficulties, related the manner in which we have convergence of the operators $\L_a^{n_j}$ to the limiting operator $\L_a^\infty$.

We begin by collecting a few lemmas related to this latter point.  The first follows from the arguments of \cite[Lemma~3.3]{KMVZZ-NLS}.

\begin{lemma}[Convergence of operators, \cite{KMVZZ-NLS}]\label{L:COO} Let $a>-(\frac{d-2}{2})^2+c_d$.  Suppose $t_n\to t\in\R$ and $y_n\to y_\infty\in\R^d$ or $|y_n|\to\infty$. Let $\L_a^n$ and $\L_a^\infty$ be as in Definition~\ref{D:Ln}.  Then the following hold:
\begin{align}
\lim_{n\to\infty}\big\| \bigl[\sqrt{\L_a^n} - \sqrt{\L_a^\infty} \,\bigr] \psi\big\|_{L^2}&=0 \qtq{for all}\psi\in \dot H^1, \label{cop5}\\
\lim_{n\to\infty}\big\| \bigl[\sqrt{\L_a^n} - \sqrt{\L_a^\infty} \,\bigr] \psi\big\|_{\dot{H}^{-1}}&=0 \qtq{for all}\psi\in L^2.\label{cop6}
\end{align}
Furthermore, if $y_\infty\neq 0$, then
\begin{equation}
\lim_{n\to\infty}\big\|\big[e^{-\L_a^n}-e^{-\L_a^\infty}\big] \delta_0\big\|_{\dot H^{-1}(\R^d)}=0. \label{cop3}
\end{equation}
\end{lemma}

For the next result, an analogous statement appears in \cite{KMVZZ-NLS} for the case of the Schr\"odinger propagator; however, the proof relies on (endpoint) Strichartz estimates and hence we need a new argument in our case.

\begin{lemma}[Convergence of operators]\label{L:COO2} Let $a>-(\frac{d-2}{2})^2$.  Suppose $y_n\to y_\infty\in\R^d$ or $|y_n|\to\infty$.  Let $\L_a^n$ and $\L_a^\infty$ be as in Definition~\ref{D:Ln}.  Then
\begin{equation}\label{strichartz-convergence}
\lim_{n\to\infty} \|(e^{it\sqrt{\L_a^n}}-e^{it\sqrt{\L_a^\infty}})\psi\|_{L_t^\infty L_x^2(\R\times\R^d)} = 0 \qtq{for all}\psi\in L^2.
\end{equation}
\end{lemma}

\begin{proof}  By approximation, it suffices to consider $\psi\in C_c^\infty(\R^d\backslash\{0\})$.  It also suffices to consider the case $|y_n|\to\infty$ or $y_n\to 0$ (by applying a fixed translation).

\emph{Case 1.} Suppose $y_n\to 0$. Then $\L_a^\infty = \L_a$.  Let us define
\[
u_n(t,x) = [e^{it\sqrt{\L_a}}\psi(\cdot-y_n)](x).
\]
Then by the triangle inequality, we have
\begin{align}
\| (e^{it\sqrt{\L_a^n}}-e^{it\sqrt{\L_a}})\psi\|_{L_t^\infty L_x^2} & \leq \|u_n(t,x+y_n)-u_n(t,x)\|_{L_t^\infty L_x^2} \label{SC1} \\
& \quad + \|e^{it\sqrt{\L_a}}[\psi(\cdot-y_n)] - e^{it\sqrt{\L_a}}\psi \|_{L_t^\infty L_x^2}. \label{SC2}
\end{align}
For \eqref{SC2}, we observe
\[
\eqref{SC2} = \|\psi(\cdot-y_n)-\psi\|_{L_x^2} = o(1)\qtq{as}n\to\infty
\]
by continuity of translation in $L^2$.  For \eqref{SC1}, we use the fundamental theorem of calculus and equivalence of Sobolev spaces to bound
\[
\|u_n(t,x+y_n)-u_n(t,x)\|_{L^2} \lesssim |y_n| \|\nabla u_n(t)\|_{L^2} \lesssim |y_n| \|\nabla \psi\|_{L^2}
\]
uniformly in $t$.  Thus \eqref{SC1} is $o(1)$ as well and the desired result follows.

\emph{Case 2.} Suppose $|y_n|\to\infty$. Then $\L_a^\infty=-\Delta$. Define the propagator
\[
S_n(t)(f,g)=\cos(t\sqrt{\L_a^n})f+\tfrac{\sin(t\sqrt{\L_a^n})}{\sqrt{\L_a^n}}g,
\]
which generates solutions to $(\partial_t^2+\L_a^n)u=0$.  Define $S_\infty(t)$ similarly. We may write
\[
e^{it\sqrt{\L_a^n}}\psi =S_n(t)(\psi,0)+iS_n(t)(0,\sqrt{\L_a^n}\psi),
\]
and similarly for $e^{it\sqrt{\L_a^\infty}}\psi$.  Thus we have
\begin{align}
(e^{it\sqrt{\L_a^n}}-e^{it\sqrt{\L_a^\infty}})\psi & = iS_n(t)(0,(\sqrt{\L_a^n}-\sqrt{\L_a^\infty}\psi) \label{E:coo1}\\
& \quad + S_n(t)(\psi,0)-S_\infty(t)(\psi,0) \label{E:coo2}  \\
& \quad + i[S_n(t)(0,\sqrt{\L_a^\infty}\psi)-S_\infty(t)(0,\sqrt{\L_a^\infty}\psi)]. \label{E:coo3}
\end{align}

Using \eqref{cop5}, we first see $\|\eqref{E:coo1}\|_{L_t^\infty L_x^2}\to 0$.

We turn to \eqref{E:coo2}; the treatment of \eqref{E:coo3} is similar. Define
\[
V_n(x) = \L_a^\infty - \L_a^n = \tfrac{a}{|x+y_n|^2}
\]
and set
\begin{equation}\label{SC3}
U_n(t) = S_n(t)(\psi,0)-S_\infty(t)(\psi,0),
\end{equation}
so that
\[
(\partial_t^2+\L_a^n)U_n = V_n(x)S_\infty(t)(\psi,0)
\]
with zero initial data. In particular, we may also write
\[
U_n(t) = \int_0^t \tfrac{\sin((t-s)\sqrt{\L_a^n})}{\sqrt{\L_a^n}}V_n(x)S_\infty(s)(\psi,0)\,ds.
\]

Using \eqref{SC3}, we can firstly observe that $u_n$ is bounded in $L_t^\infty \dot H_x^{0-}$.  Thus, it suffices to prove that $u_n$ tends to zero in $L_t^\infty \dot H_x^{\frac12}$.  For this, we use the Duhamel formulation and use Strichartz to estimate
\[
\|U_n\|_{L_t^\infty \dot H_x^{\frac12}} \lesssim \|V_n S_\infty(\psi,0)\|_{L_{t,x}^{\frac{2(d+1)}{d+3}}}.
\]
Recalling that $\psi\in C_c^\infty$ and that $u(t,x):=S_\infty(\psi,0)$ solves the (free) linear wave equation, we are left to prove that
\[
\lim_{n\to\infty} \| |x+y_n|^{-2} u(t,x)\|_{L_{t,x}^{\frac{2(d+1)}{d+3}}} =0
\]
on $[0,\infty)\times\R^d$ (say), where $u$ satisfies the following:
\begin{itemize}
\item $u(t)$ is supported in the ball of radius $t+C$ for some $C>0$,
\item $u(t)$ is uniformly bounded in $L_x^2$,
\item $u(t)$ decays like $t^{-\frac{d-1}{2}}$ in $L_x^\infty$.
\end{itemize}

Let $0<\eps\ll 1$.  We first consider the contribution of $0<t<\eps|y_n|$. By the support properties of $u$, we have $|x+y_n|^{-2}\lesssim |y_n|^{-2}$ in this region. Thus
\begin{align*}
\||x+y_n|^{-2}&u\|_{L_{t,x}^{\frac{2(d+1)}{d+3}}(\{|t|<\eps|y_n|\}\times\R^d)}  \\
 & \lesssim |y_n|^{-2} \|\langle t\rangle^{\frac{d}{d+1}}\|_{L_t^{\frac{2(d+1)}{d+3}}(\{|t|<\eps|y_n|\})} \|u\|_{L_t^\infty L_x^2} \lesssim |y_n|^{-\frac12},
\end{align*}
which is acceptable.

We next consider the contribution of $t>\eps|y_n|$.  We split the spatial integral into the regions where $|x+y_n|\geq 1$ and $|x+y_n|<1$, respectively.  The estimates on each region are similar, so let us consider the first case.  Using the decay properties of $u(t)$, we have
\begin{align*}
\| |x+y_n|^{-2}&u\|_{L_{t,x}^{\frac{2(d+1)}{d+3}}(\{|t|>\eps|y_n|\}\times\{|x+y_n|\geq 1\})} \\
& \lesssim \bigl\| \||x|^{-2}\|_{L_x^{\frac{d}{2}+}(\{|x|>1\})} \|u(t)\|_{L_x^{\frac{2d(d+1)}{d^2-d-4}-}} \bigr\|_{L_t^{\frac{2(d+1)}{d+3}}(\{|t|>\eps|y_n|\})} \\
& \lesssim \bigl\| |t|^{-\frac{(d-1)(d+2)}{d(d+1)}+} \bigr\|_{L_t^{\frac{2(d+1)}{d+3}}(\{|t|>\eps|y_n|\})} \\
& \lesssim (\eps |y_n|)^{-\frac{(d-1)(d+2)}{d(d+1)}+\frac{d+3}{2(d+1)}+},
\end{align*}
which is acceptable.  As the modifications necessary to treat $|x+y_n|<1$ are straightforward (e.g. we put $|x|^{-2}$ in $L^{\frac{d}{2}-}$), this completes the proof. \end{proof}

We now record a few corollaries that will be of use below.  We first have the following.

\begin{corollary}\label{C:COO}  Let $a>-(\tfrac{d-2}{2})^2+c_d$.  Suppose $y_n\to y_\infty\in\R^d$ or $|y_n|\to\infty$ and let $\L_a^n,\L_a^\infty$ be as in Definition~\ref{D:Ln}.  Then
\[
\lim_{n\to\infty} \|[e^{it\sqrt{\L_a^n}}-e^{-it\sqrt{\L_a^\infty}}]\psi\|_{S(\R)}=0 \qtq{for all}\psi\in \dot H^1,
\]
where $S(\cdot)$ is as in \eqref{Snorm}.
\end{corollary}

\begin{proof}  Let us show the proof for the $L_{t,x}^{\frac{2(d+1)}{d-2}}$ component of the $S$-norm.

By Strichartz, the quantity in question is finite for $\psi\in \dot H^1$.  Thus, we can reduce to the case of $\psi\in C_c^\infty(\R^d\backslash\{y_\infty\})$ (if $y_n\to y_\infty$) or $C_c^\infty(\R^d)$ (if $|y_n|\to\infty$).  For such $\psi$, we use H\"older to estimate
\begin{align*}
\|[e^{it\sqrt{\L_a^n}}-e^{-it\sqrt{\L_a^\infty}}]\psi\|_{L_{t,x}^{\frac{2(d+1)}{d-2}}} & \lesssim \|[e^{it\sqrt{\L_a^n}}-e^{-it\sqrt{\L_a^\infty}}]\psi\|_{L_t^\infty L_x^2}^\theta \\
& \quad \times  \|[e^{it\sqrt{\L_a^n}}-e^{-it\sqrt{\L_a^\infty}}]\psi\|_{L_t^q L_x^r}^{1-\theta},
\end{align*}
where $\theta\in(0,1)$ and
\[
\tfrac{\theta}{2}+\tfrac{1-\theta}{r}=\tfrac{d-2}{2(d+1)},\quad \tfrac{1-\theta}{q}=\tfrac{d-2}{2(d+1)}.
\]
Choosing $\theta$ sufficiently small so that we may apply Stirchartz and the equivalence of Sobolev spaces holds, we can estimate
\[
\|[e^{it\sqrt{\L_a^n}}-e^{-it\sqrt{\L_a^\infty}}]\psi\|_{L_t^q L_x^r} \lesssim \| |\nabla|^{\frac{d}{2}-\frac{1}{q}-\frac{d}{r}}\psi\|_{L^2} \lesssim 1,
\]
so that the result follows from Lemma~\ref{L:COO2}. \end{proof}

Next, we have the following, which is completely analogous to equation (3.4) in \cite{KMVZZ-NLS}.
\begin{corollary} Let $a>-(\frac{d-2}{2})^2+c_d$.  Suppose $y_n\to y_\infty\in\R^d$ or $|y_n|\to\infty$.  Let $\L_a^n$ and $\L_a^\infty$ be as in Definition~\ref{D:Ln}. If $t_n\to t\in\R$, then
\begin{equation}
\lim_{n\to\infty}\bigl\| [e^{-it_n\sqrt{\L_a^n}}-e^{-it\sqrt{\L_a^\infty}}]\psi\|_{\dot H^{-1}} = 0 \qtq{for all}\psi\in \dot H^{-1}.\label{cop-new}
\end{equation}
\end{corollary}

\begin{proof} By the equivalence of Sobolev spaces, we may write $\psi\in\dot H^{-1}$ in the form $\sqrt{\L_a^\infty}\phi$ for some $\phi\in L^2$.  We write
\begin{align*}
[e^{it_n\sqrt{\L_a^n}}-e^{it\sqrt{\L_a^\infty}}]\sqrt{\L_a^\infty}\phi & = e^{it_n\sqrt{\L_a^n}}\bigl[\sqrt{\L_a^\infty}-\sqrt{\L_a^n}]\phi \\
& \quad + \sqrt{\L_a^n}[e^{it_n\sqrt{\L_a^n}}-e^{it_n\sqrt{\L_a^\infty}}]\phi \\
& \quad + \sqrt{\L_a^n}[e^{it_n\sqrt{\L_a^\infty}}-e^{it\sqrt{\L_a^\infty}}]\phi \\
& \quad + \bigl[\sqrt{\L_a^n}-\sqrt{\L_a^\infty}\bigr]e^{it\sqrt{\L_a^\infty}}\phi.
\end{align*}
Applying \eqref{cop6} to the first and last terms and applying Lemma~\ref{L:COO2} to the second term, we find that
\[
\limsup_{n\to\infty} \|[e^{it_n\sqrt{\L_a^n}}-e^{it\sqrt{\L_a^\infty}}]\psi\|_{\dot H^{-1}} \leq \limsup_{n\to\infty}\|[e^{it_n\sqrt{\L_a^\infty}}-e^{it\sqrt{\L_a^\infty}}]\phi\|_{L^2},
\]
which vanishes by the spectral theorem.  This completes the proof.  \end{proof}

The next result will be important proving energy decoupling in the profile decomposition.

\begin{corollary}\label{C:CooL4} Let $a>-(\frac{d-2}{2})^2+c_d$ and $\psi\in \dot H^1$.  Given a sequence $t_n\to\pm\infty$ and any sequence $\{y_n\}\subset\R^d$, we have
\[
\lim_{n\to\infty} \| e^{it_n\sqrt{\L_a^n}}\psi\|_{L_x^\frac{2d}{d-2}}=0,
\]
where $\L_a^n$ is as in Definition~\ref{D:Ln}.
\end{corollary}

\begin{proof} Without loss of generality, we assume $y_n\to y_\infty\in\R^d$ or $|y_n|\to\infty$.  We let $\L_a^\infty$ be as in Definition~\ref{D:Ln}.

We begin by using Sobolev embedding to estimate
\begin{align}
\| e^{it_n\sqrt{\L_a^n}}\psi\|_{L_x^{\frac{2d}{d-2}}} & \lesssim \|\sqrt{\L_a^\infty}\bigl[e^{it_n\sqrt{\L_a^n}}-e^{it_n\sqrt{\L_a^\infty}}\bigr]\psi\|_{L_x^2} \label{L41}\\
& \quad + \|e^{it_n\sqrt{\L_a^\infty}}\psi\|_{L_x^{\frac{2d}{d-2}}}.\label{L42}
\end{align}
To estimate \eqref{L41}, we first use the triangle inequality and find
\begin{align*}
\|\sqrt{\L_a^\infty}\bigl[e^{it_n\sqrt{\L_a^n}}\psi-e^{it_n\sqrt{\L_a^\infty}}\psi\bigr]\|_{L_x^2} &
\lesssim \|\bigl[\sqrt{\L_a^\infty}-\sqrt{\L_a^n}\bigr]e^{it_n\sqrt{\L_a^n}}\psi\|_{L_x^2} \\
& \quad + \|e^{it_n\sqrt{\L_a^n}}[\sqrt{\L_a^n}-\sqrt{\L_a^\infty}]\psi\|_{L_x^2} \\
& \quad + \|[e^{it_n\sqrt{\L_a^n}} - e^{it_n\sqrt{\L_a^\infty}}]\sqrt{\L_a^\infty}\psi\|_{L_x^2}.
\end{align*}
The second term is $o(1)$ as $n\to\infty$ by \eqref{cop5}.  The third term is $o(1)$ as $n\to\infty$ by Lemma~\ref{L:COO2} (bounding an individual $t_n$ by the $L^\infty$ norm in time).  Thus we need to show that the first term is $o(1)$ as $n\to\infty$, as well.  To see this, we use duality to write
\begin{align*}
\|[\sqrt{\L_a^\infty}-\sqrt{\L_a^n}]e^{it_n\sqrt{\L_a^n}}\psi\|_{L_x^2}&=\sup\bigl| \bigl\langle e^{it_n\sqrt{\L_a^n}}\psi,[\sqrt{\L_a^\infty}-\sqrt{\L_a^n}]g\rangle \bigl|\\
& \lesssim \|[ \sqrt{\L_a^\infty}-\sqrt{\L_a^n}]g\|_{\dot H_x^{-1}}\|\psi\|_{\dot H_x^1},
\end{align*}
where the supremum is over $g\in L^2$ with $\|g\|_{L^2}=1$.  The claim now follows from \eqref{cop6}.

It remains to estimate \eqref{L42}.  By density, we may assume $\psi\in C_c^\infty(\R^d\backslash\{y_\infty\})$ if $y_n\to y_\infty$ and $\psi\in C_c^\infty$ if $|y_n|\to\infty$. Writing
\[
F(t)=\|e^{it\sqrt{\L_a^\infty}}\psi\|_{L_x^{\frac{2d}{d-2}}},
\]
we have by Strichartz that $F\in L_t^{q}(\R)$ for sufficiently large $q<\infty$.  Furthermore, $F$ is Lipschitz; indeed, by Sobolev embedding
\[
|\partial_t F(t)| \leq \| \partial_t e^{it\sqrt{\L_a^\infty}}\psi\|_{L_x^{\frac{2d}{d-2}}} \lesssim \| \sqrt{\L_a^\infty} \psi\|_{\dot H_x^1} \lesssim 1.
\]
Thus $F(t_n)\to 0$ as $n\to\infty$.  This completes the proof.
\end{proof}

We turn now to the linear profile decomposition, Proposition~\ref{P:LPD}.  As mentioned above, the starting point is a refined Strichartz estimate for identifying a scale at which concentration occurs.  We have the following:

\begin{lemma}[Refined Strichartz estimate]\label{L:refined} There exists $\theta\in(0,1)$ so that
\[
\|e^{-it\sqrt{\L_a}} f\|_{L_{t,x}^{\frac{2(d+1)}{d-2}}} \lesssim \|f\|_{\dot H_x^1}^{1-\theta}\sup_{N\in2^{\mathbb{Z}}} \| e^{-it\sqrt{\L_a}}P_N^a f\|_{L_{t,x}^{\frac{2(d+1)}{d-2}}}^{\theta}.
\]
\end{lemma}

\begin{proof} Denote $f_N=P_N^af$, $u(t)=e^{-it\sqrt{\L_a}}f$, $u_N=P_N^a u$, and so on. Let us also write $r=\frac{2(d+1)}{d-2}$. By the square function estimate and Bernstein (see Proposition~\ref{P:HAT}), as well as Strichartz, we have
\begin{align*}
\|u\|_{L_{t,x}^r}^r & \lesssim \iint \biggl(\sum_N |u_N|^2\biggr)^{\frac{r}{2}}\,dx\,dt \\
& \lesssim \|\bigl(\sum_N |u_N|^2\bigr)^{\frac{1}{2}}\|_{L_{t,x}^r}^{r-4}\sum_{N_1\leq N_2} \|u_{N_1}\|_{L_t^r L_x^{r+}}\|u_{N_1}\|_{L_{t,x}^r} \|u_{N_2}\|_{L_{t,x}^r}\|u_{N_2}\|_{L_t^r L_x^{r-}} \\
& \lesssim \|u\|_{L_{t,x}^r}^{r-4}\bigl[\sup_N \|u_N\|_{L_{t,x}^r}\bigr]^2\sum_{N_1\leq N_2}N_1^{0+}\|u_{N_1}\|_{L_{t,x}^r}\|f_{N_2}\|_{\dot H_x^{1-}} \\
& \lesssim \|f\|_{\dot H^1}^{r-4}\bigl[\sup_N \|u_N\|_{L_{t,x}^r}\bigr]^2 \sum_{N_1\leq N_2}\bigl(\tfrac{N_1}{N_2}\bigr)^{0+}\|f_{N_1}\|_{\dot H_x^1}\|f_{N_2}\|_{\dot H_x^1}.
\end{align*}
Applying Cauchy--Schwarz, the result now follows with $\theta=\tfrac{2}{r}$.
\end{proof}

The next ingredient for the linear profile decomposition is the following inverse Strichartz estimate, which demonstrates how to remove each bubble of concentration.

\begin{proposition}[Inverse Strichartz]\label{P:inverse} Let $a>-(\frac{d-2}{2})^2 +c_d$ and suppose $f_n\in \dot H^1$ satisfy
\[
\lim_{n\to\infty}\|f_n\|_{\dot H^1}=A<\infty\qtq{and} \|e^{-it\sqrt{\L_a}}f_n\|_{L_{t,x}^{\frac{2(d+1)}{d-2}}} =\eps>0.
\]
Passing to a subsequence, there exist $\phi\in \dot H^1$, $N_n\in 2^{\mathbb{Z}}$, and $(t_n,x_n)\in\R^{1+d}$ such that
\begin{align}
& g_n(x) = N_n^{-(\frac{d}{2}-1)}[e^{-it_n\sqrt{\L_a}} f_n]\big(\tfrac{x}{N_n}+x_n\big) \rightharpoonup \phi(\cdot) \qtq{weakly in}\dot{H}_x^1,\label{weak}   \\
 & \|\phi\|_{\dot{H}_a^1} \gtrsim
\varepsilon(\tfrac{\varepsilon}{A})^{c}.\label{LB}
\end{align}
Furthermore, defining
\[
\phi_n(x) = N_n^{\frac{d}{2}-1}e^{it_n\sqrt{\L_a}}[\phi(N(x-x_n))]=N_n^{\frac{d}{2}-1}[e^{i N_n t_n\sqrt{\L_a^n}}\phi](N_n(x-x_n)),
\]
where $\L_a^n$ is as in Definition~\ref{D:Ln} with $y_n=N_n x_n$, we have
\begin{align}
& \lim_{n\to\infty} \bigl\{ \|f_n\|_{\dot{H}^1_a}^2 - \|f_n-\phi_n\|_{\dot{H}^1_a}^2 - \|\phi_n\|_{\dot{H}^1_a}^2 \bigr\} =0, \label{H1-decouple}
\end{align}
and
\begin{equation}
\lim_{n\to\infty}\Bigl\{ \|f_n\|_{L^{\frac{2d}{d-2}}_x}^{\frac{2d}{d-2}}-\|f_n-\phi_n\|_{L^{\frac{2d}{d-2}}_x}^{\frac{2d}{d-2}}-\|\phi_n\|_{L^{\frac{2d}{d-2}}_x}^{\frac{2d}{d-2}}\Bigr\} =0,\label{dect}
\end{equation}
Finally, we may assume that either $N_nt_n\to\pm\infty$ or $t_n\equiv 0$, and that either $N_n|x_n|\to \infty$ or $x_n\equiv 0$.
\end{proposition}

\begin{proof} Let $r=\frac{2(d+1)}{d-2}$.  We use $c$ to denote a positive constant that may change throughout the proof.  Using Lemma~\ref{L:refined}, there exists $N_n$ such that
\[
\|e^{-it\L_a}P_{N_n}^a f_n\|_{L_{t,x}^r}\gtrsim \eps(\tfrac{\eps}{A})^c.
\]
Using H\"older followed by Bernstein, we have
\[
\|P_N^a F\|_{L_x^r(|x|\leq C N^{-1})} \lesssim C^{0+}\|P_N^a F\|_{L_x^r},
\]
and hence for $C$ sufficiently small we have
\[
\| e^{-it\L_a}P_{N_n}^a f_n\|_{L_{t,x}^r(\R\times\{|x|>CN_n^{-1}\})}\gtrsim \eps(\tfrac{\eps}{A})^c.
\]
Thus, applying H\"older, Strichartz, and Bernstein, we deduce
\begin{align*}
\eps(\tfrac{\eps}{A})^c & \lesssim \|e^{-it\sqrt{\L_a}}P_{N_n}^af_n\|_{L_{t,x}^{(1-\theta)r}}^{1-\theta} \|e^{-it\sqrt{\L_a}}P_{N_n}^a f_n\|_{L_{t,x}^\infty}^{\theta} \\
& \lesssim N_n^{-\theta[\frac{d}{2}-1]}A^{1-\theta}\| e^{-it\sqrt{\L_a}}P_{N_n}^a f_n\|_{L_{t,x}^\infty}^\theta
\end{align*}
for small $\theta>0$.  It follows that there exist $(\tau_n,x_n)$ with $|x_n|N_n\geq C$ and
\begin{equation}\label{LB0}
N_n^{-(\frac{d}{2}-1)}\bigl| (e^{-i\tau_n\sqrt{\L_a}}P_{N_n}^a f_n)(x_n)\bigr| \gtrsim \eps(\tfrac{\eps}{A})^c.
\end{equation}
Passing to a subsequence, we may assume $N_n\tau_n\to \tau_\infty\in[-\infty,\infty]$.  If $\tau_\infty$ is finite, define $t_n\equiv 0$; otherwise, let $t_n=\tau_n$.

We now let
\[
g_n(x):= N_n^{-(\frac{d}{2}-1)}[e^{it_n\sqrt{\L_a}}f_n](\tfrac{x}{N_n}+x_n).
\]
Note that
\[
\|g_n\|_{\dot H^1} \lesssim \|f_n\|_{\dot H^1} \lesssim A.
\]
Therefore there exists $\phi\in \dot H^1$ so that $g_n\rightharpoonup \phi$ weakly in $\dot H^1$, yielding \eqref{weak}.  Expanding inner products and appealing to Lemma~\ref{L:COO}, we can also deduce \eqref{H1-decouple}.

We turn to \eqref{LB}.  We now wish to define $h_n$ so that
\[
|\langle g_n,h_n\rangle| = N_n^{-(\frac{d}{2}-1)}|(e^{-i\tau_n\sqrt{\L_a}}P_{N_n}^a f_n)(x_n)| \gtrsim \eps(\tfrac{\eps}{A})^c.
\]
A computation shows that we should take
\[
h_n = e^{iN_n(\tau_n-t_n)\sqrt{\L_a^n}} P_1^n \delta_0
\]
where $ P_1^n = e^{-\L_a^n}-e^{-4\L_a^n}$ and $\L_a^n$ is as in Definition~\ref{D:Ln} with $y_n=N_n x_n$. Using \eqref{cop3} and \eqref{cop-new}, we find that
\[
h_n \to h_\infty:=\begin{cases} P_1^\infty\delta_0 & \tau_\infty\in\{\pm\infty\}, \\ e^{-i\tau_\infty\sqrt{\L_a^\infty}} P_1^\infty\delta_0& \tau_\infty\in\R \end{cases}
\]
strongly in $\dot H^{-1}$, where $P_1^\infty = e^{-\L_a^\infty}-e^{-4\L_a^\infty}$.  Therefore, by strong convergence of $h_n$ and weak convergence of $g_n$, we can conclude that
\[
\eps(\tfrac{\eps}{A})^c \lesssim \|\phi\|_{\dot H^1}\| h_\infty\|_{\dot H^{-1}}.
\]
Using the heat kernel estimates (cf. \eqref{heat}), the embedding $L^{\frac{2d}{d+2}}\hookrightarrow \dot H^{-1}$, and $N_n|x_n|\gtrsim c$, we can show that
\[
\|h_\infty\|_{\dot H^{-1}} \lesssim 1,
\]
and hence \eqref{LB} holds.

Finally, we turn to \eqref{dect}.  If $t_n\equiv 0$, then using Rellich--Kondrashov (to get $g_n\to\phi$ a.e.) and Lemma~\ref{L:fatou} we get
\[
\lim_{n\to\infty} \biggl[\|g_n\|_{L_x^{\frac{2d}{d-2}}}^{\frac{2d}{d-2}}-\|g_n-\phi\|_{L_x^{\frac{2d}{d-2}}}^{\frac{2d}{d-2}}-\|\phi\|_{L_x^{\frac{2d}{d-2}}}^{\frac{2d}{d-2}}\biggr]=0,
\]
which yields \eqref{dect} after a change of variables. In the case that $t_n=\tau_n$, the result follows from the fact that $\phi_n\to 0$ in $L^{\frac{2d}{d-2}}$ (by Corollary~\ref{C:CooL4}).

Finally, by passing to a further subsequence, we can assume that either $N_n|x_n|\to\infty$ or $N_n x_n\to y_\infty\in\R^d$.  In the latter case, we may take $x_n\equiv 0$ by replacing $\phi$ with $\phi(\cdot-y_\infty)$.
\end{proof}

We now turn to the proof of the linear profile decomposition Proposition~\ref{P:LPD}.  As the proof follows along well-established lines, we will be somewhat brief.

\begin{proof}[Proof of Proposition~\ref{P:LPD}] The decompositon \eqref{E:LP0} and the decouplings \eqref{E:LP2} and \eqref{E:LP6} follow by induction.  One sets $r_n^0=f_n$ and applies Proposition~\ref{P:inverse} to  the sequence $r_n^0$ to find $\phi_n^1$ (and we set $\lambda_n^1=[N_n^1]^{-1}$); one then applies Proposition~\ref{P:inverse} to the sequences $r_n^J:=r_n^{J-1}-\phi_n^J$.  The process terminates at a finite $J^*$ if $\|e^{-it\sqrt{\L_a}}r_n^{J^*}\|_{L_{t,x}^{\frac{2(d+1)}{d-2}}}=0$.

Defining
\[
\eps_J = \lim_{n\to\infty} \|e^{-it\sqrt{\L_a}}r_n^J\|_{L_{t,x}^{\frac{2(d+1)}{d-2}}}\qtq{and} A_J = \lim_{n\to\infty} \|r_n^J\|_{\dot H^1},
\]
we have that $\eps_J\to 0$ as a consequence of \eqref{E:LP2}, \eqref{LB}, and $A_J\leq A_0$; in fact,
\[
\sum_{j=1}^J \eps_{j-1}^2\bigl(\tfrac{\eps_{j-1}}{A_0}\bigr)^{c} \lesssim \sum_{j=1}^J\|\phi_n^j\|_{\dot H_a^1}^2 \lesssim A_0^2.
\]

By construction and \eqref{weak}, we have
\begin{equation}\label{E:weak2}
(\lambda_n^J)^{\frac{d}{2}-1}[e^{-it_n^J\sqrt{\L_a}}r_n^{J-1}](\lambda_n^J x+x_n^J\bigr)\rightharpoonup \phi^J\qtq{for each finite}J\geq 1.
\end{equation}
Recalling that $r_n^J=r_n^{J-1}-\phi^J$, we deduce \eqref{weaklpd}.  This will also play an important role in proving the orthogonality condition \eqref{E:LP5}, to which we now turn.

Putting $(j,k)$ in lexicographical order, we suppose toward a contradiction that \eqref{E:LP5} fails for the first time at some $(j,k)$ with $j<k$.  Thus
\begin{equation}\label{lpd-contra0}
\tfrac{\lambda_n^j}{\lambda_n^k}\to\lambda_0,\quad \tfrac{x_n^j-x_n^k}{\sqrt{\lambda_n^j \lambda_n^k}}\to y_0,\qtq{and} \tfrac{t_n^j-t_n^k}{\sqrt{\lambda_n^j \lambda_n^k}}\to t_0,
\end{equation}
but \eqref{E:LP5} holds for every pair $(j,\ell)$ with $j<\ell<k$.  Now, using \eqref{E:LP0} to get an expression for both $r_n^j$ and $r_n^{k-1}$, we have
\[
r_n^j -\sum_{\ell=j+1}^{k-1}\phi_n^\ell = r_n^{k-1}.
\]
Therefore, using \eqref{E:weak2}, we have
\begin{equation}\label{lpd-contra}
(\lambda_n^k)^{\frac{d}{2}-1}[e^{-it_n^k\sqrt{\L_a}}r_n^j](\lambda_n^kx+x_n^k)-\sum_{\ell=j+1}^{k-1}(\lambda_n^k)^{\frac{d}{2}-1}[e^{-it_n^k\sqrt{\L_a}}\phi_n^\ell](\lambda_n^kx+x_n^k)\rightharpoonup \phi^k(x).
\end{equation}
To get a contradiction, we will show that both of the terms above converge weakly to zero, contradicting that $\phi^k$ is nontrivial.

For the first term in \eqref{lpd-contra}, we will use \eqref{lpd-contra0}.  Let us introduce the notation
\[
(g_n^j)^{-1}f(x) = (\lambda_n^j)^{\frac{d}{2}-1}f(\lambda_n^j x + x_n^k).
\]
We rewrite the first term in \eqref{lpd-contra} as
\[
(g_n^k)^{-1}e^{i(t_n^j-t_n^k)\sqrt{\L_a}}g_n^j\bigl[(g_n^j)^{-1}e^{-it_n^j\sqrt{\L_a}}r_n^j\bigr],
\]
and observe that the term inside the square brackets converges weakly to zero.  The operator preceding the square brackets can be rewritten
\[
(g_n^k)^{-1}g_n^j e^{i(\lambda_n^j)^{-1}(t_n^j-t_n^k)\sqrt{\L_a^{n_j}}},
\]
where $\L_a^{n_j}$ is as in Definition~\ref{D:Ln} with the sequence $y_n=(\lambda_n^j)^{-1}x_n^j$.  Noting that \eqref{lpd-contra0} implies that the adjoint of $(g_n^k)^{-1}g_n^j$ converges strongly and that the sequence $(\lambda_n^j)^{-1}(t_n^j-t_n^k)$ converges to some finite real number, the claim now reduces to the following lemma.  This lemma (and its proof) is completely analogous to \cite[Lemma~3.8]{KMVZZ-NLS}

\begin{lemma}\label{L:weak1} Suppose $f_n\in\dot H^1$ converges to zero weakly in $\dot H^1$ and $t_n\to t_\infty\in\R$.  Then for any $y_n\in\R^d$,
\[
e^{-it_n\sqrt{\L_a^n}}f_n\rightharpoonup 0 \qtq{weakly in}\dot H^1,
\]
where $\L_a^n$ is as in Definition~\ref{D:Ln} with the sequence $y_n$.
\end{lemma}

\begin{proof} Without loss of generality, we assume $y_n\to y_\infty\in\R^d$ or $|y_n|\to\infty$.  We let $\L_a^\infty$ be as in Definition~\ref{D:Ln}.

We claim that it suffices to prove
\begin{equation}\label{claimweak}
 e^{-it_\infty\sqrt{\L_a^n}} f_n\rightharpoonup 0\qtq{weakly in}\dot H^1.
 \end{equation}
To see this, given $\psi\in C_c^\infty(\R^d\backslash\{y_\infty\})$ (if $y_n\to y_\infty$) or $\psi\in C_c^\infty(\R^d)$ (if $|y_n|\to\infty$), we estimate
\begin{align*}
\bigl|\bigl\langle [e^{-it_n\sqrt{\L_a^n}}-e^{-it_\infty \sqrt{\L_a^n}}]f_n, \psi\bigr\rangle_{\dot H^1_x}\bigr|
&\lesssim \big\|[e^{-it_n\sqrt{\L_a^n}}-e^{-it_\infty \sqrt{\L_a^n}}]f_n\big\|_{L^2_x} \|\Delta\psi\|_{L^2_x}\\
&\lesssim |t_n-t_\infty| \|\sqrt{\L_a^n}f_n\|_{L^2_x}\|\Delta\psi\|_{L^2_x},
\end{align*}
where we have used the spectral theorem and the simple inequality
\[
|e^{-it_n\sqrt{\lambda}}-e^{-it_\infty\sqrt{\lambda}}| \lesssim |t_n-t_\infty|\sqrt{\lambda}
\]
for $\lambda\geq 0$.  Thus the claim follows.  To prove \eqref{claimweak}, we take $\psi$ as above and begin by estimating
\begin{align*}
|\langle e^{-it_\infty\sqrt{\L_a^n}}f_n,\psi\rangle_{\dot H^1}| & \lesssim |\langle  f_n, [e^{it_\infty\sqrt{\L_a^n}}-e^{it_\infty\sqrt{\L_a^\infty}}](-\Delta\psi)\rangle_{L^2}| \\
& \quad + |\langle f_n, e^{it_\infty\sqrt{\L_a^\infty}}(-\Delta\psi)\rangle_{L^2}|.
\end{align*}
The first term on the right-hand side converges to zero by \eqref{cop-new}, using the fact that $\Delta\psi \in \dot H^{-1}$, while the second term converges to zero due to the weak convergence of $f_n$.  This completes the proof. \end{proof}

We turn now to the second term in \eqref{lpd-contra}.  This time we take a similar approach, relying on the fact that \eqref{E:LP5} holds for each pair $(j,\ell)$ with $j<\ell<k$.  Omitting some of the details, the claim boils down to the following lemma.  This lemma (and its proof) is again completely analogous to \cite[Lemma~3.9]{KMVZZ-NLS}.

\begin{lemma}\label{L:weak2} Let $f\in\dot H^1$ and let $(t_n,x_n)\in\R^{1+d}$ and $y_n\in\R^d$.  Writing $\L_a^n$ as in Definition~\ref{D:Ln} with the sequence $y_n$, we have
\[
[e^{-it_n\sqrt{\L_a^n}}f](x+x_n)\rightharpoonup 0 \qtq{weakly in}\dot H^1
\]
whenever $|t_n|\to\infty$ or $|x_n|\to\infty$.
\end{lemma}

\begin{proof} Without loss of generality, assume $y_n\to y_\infty\in\R^d$ or $|y_n|\to\infty$.  Take $\L_a^\infty$ as in Definition~\ref{D:Ln}.

Suppose $t_n\to\infty$; the case $t_n\to-\infty$ is similar.  We let $\psi\in C_c^\infty(\R^d\backslash\{y_\infty\})$ (if $y_n\to y_\infty$) or $C_c^\infty(\R^d)$ (if $|y_n|\to\infty$).  Define
\[
F_n(t) = \langle e^{-it\sqrt{\L_a^n}}f](x+x_n),\psi\rangle_{\dot H^1}.
\]
We need to show that $F_n(t_n)\to 0$ as $n\to\infty$.  To this end, we first compute the time derivative and observe that $|\partial_t F_n|\lesssim 1$ uniformly in $n$.  Thus, letting $r=\tfrac{2(d+1)}{d-2}$, we have by the fundamental theorem of calculus that
\[
|F_n(t_n)|^{r+1} \lesssim |F_n(t_0)|^{r+1} + \| F_n\|_{L_t^{r}(t_n,t_0)}^r\qtq{for any}t_0>t_n.
\]
In particular, it suffices to show that each $F_n \in L_t^r$ (which yields $F_n\to 0$ as $t\to\infty$ for each fixed $n$) and that
\[
\lim_{n\to\infty} \|F_n\|_{L_t^r(t_n,\infty)}=0.
\]
That $F_n \in L_t^r$ follows from H\"older's inequality and Strichartz.  For the second point, we estimate by H\"older's inequality
\begin{align*}
\|F_n\|_{L_t^r([t_n,\infty])} & \lesssim \| [e^{-it\sqrt{\L_a^n}}-e^{-it\sqrt{\L_a^\infty}}]f\|_{L_{t,x}^r([t_n,\infty]\times\R^d)} \\
& \quad + \|e^{-it\sqrt{\L_a^\infty}} f\|_{L_{t,x}^r([t_n,\infty)\times\R^d)}.
\end{align*}
The first term converges to zero by Corollary~\ref{C:COO}, while the second term tends to zero as $n\to\infty$ by Strichartz and the monotone convergence.  This completes the proof in the case $t_n\to\infty$.

Finally, suppose $t_n$ is bounded (and $t_n\to t_\infty$, say) but $|x_n|\to\infty$.  In this case, we can move the translation inside appeal to Lemma~\ref{L:weak1}, cf.
\[
[e^{-it_n\sqrt{\L_a^n}}f](\cdot+x_n) = e^{-it\sqrt{\tilde \L_a^n}}[f(\cdot + x_n)]
\]
where $\tilde \L_a^n$ is as in Definition~\ref{D:Ln} with the sequence $x_n+y_n$. This completes the proof. \end{proof}

With Lemma~\ref{L:weak1} and Lemma~\ref{L:weak2} in place, we complete the proof of \eqref{E:LP5} and hence the proof of Proposition~\ref{P:LPD}.
\end{proof}

\section{Existence of minimal blowup solutions}\label{S:exist}

In this section, we prove that if Theorem~\ref{T:defocusing} or Theorem~\ref{T:focusing} fails, then we can construct minimal blowup solutions.

We then prove the existence of scattering nonlinear profiles associated to linear profiles with translation parameters tending to infinity (Proposition~\ref{P:NP}).  With these two ingredients in place, we can then follow fairly standard arguments to deduce the existence of minimal blowup solutions (see Theorem~\ref{T:reduction}).  Finally, arguments from \cite{KenigMerle} will allow us to further reduce the class of solutions under consideration (see Theorem~\ref{T:refine}).

We recall the mapping $T_a$ introduced in \eqref{T}, which takes a pair of real-valued functions and returns a single complex-valued function through
\[
T_a(f,g)=f+i\L_a^{-\frac12}g.
\]
We also recall the notation $\tilde E_a$ from \eqref{Etilde}.  Note that
\[
T_0(f,g)=f+i|\nabla|^{-1}g.
\]

\subsection{Construction of nonlinear profiles}\label{S:embedding}

We will construct nonlinear profiles via approximation by solutions to the free nonlinear wave equation.  To construct scattering solutions to the free NLW, we rely on the result of \cite{KenigMerle}.

\begin{theorem}[Scattering for the free NLW, \cite{KenigMerle, BahGer}]\label{T:NLW} Let $(w_0,w_1)\in\dot H^1\times L^2$ and $\mu\in\{\pm1\}$.  If $\mu=-1$, assume further that
\[
E_0[(w_0,w_1)]<E_0[W_0] \qtq{and}\|w_0\|_{\dot H^1} < \|W_0\|_{\dot H^1}.
\]
There exists a unique global solution $w$ to
\begin{equation}\label{nlw0}
\partial_t^2 w - \Delta w + \mu |w|^{\frac{4}{d-2}} w =0
\end{equation}
that scatters in both time directions and obeys global $L_{t,x}^{\frac{2(d+1)}{d-2}}$ space-time bounds.

Furthermore, given $(w_0,w_1)\in\dot H^1\times\dot L^2$ satisfying
\[
\tfrac12 \| w_0\|_{\dot H_x^1}^2+\tfrac12\|w_1\|_{L_x^2}^2 < E_0[W_0]\qtq{and} \|w_0\|_{\dot H_x^1} <\|W_0\|_{\dot H^1}
\]
in the case $\mu=-1$, there exists a unique global solution to \eqref{nlw0} that scatters to $(w_0,w_1)$ as $t\to\infty$ (or as $t\to-\infty$).
\end{theorem}

We turn to the main result of this section.  We assume $d\in\{3,4\}$ and $a>-(\tfrac{d-2}{2})^2+c_d$, as usual.  In light of the application below, we will state the following result in terms of constructing scattering solutions to \eqref{Tnlw}, rather than the original equation \eqref{nlw}.

\begin{proposition}[Construction of nonlinear profiles]\label{P:NP} Suppose $t_n\in\R$ satisfy $t_n\equiv 0$ or $t_n\to\pm\infty$, and suppose $x_n\in\R^d$ satisfy $|x_n|\lambda_n^{-1}\to\infty$.  Let $\phi\in\dot H^1$ and define
\[
\phi_n(x) = \lambda_n^{-(\frac{d}{2}-1)}[e^{-it_n\sqrt{\L_a^n}}\phi](\tfrac{x-x_n}{\lambda_n}),
\]
where $\L_a^n$ is as in Definition~\ref{D:Ln} with the sequence $y_n=\lambda_n^{-1} x_n$.
\begin{itemize}
\item If $\mu=+1$ (the defocusing case), then for $n$ sufficiently large there exists a global solution $v_n$ to \eqref{Tnlw} with $v_n(0)=\phi_n$ satisfying
\[
\|v_n\|_{\dot S^1(\R)} \lesssim1,
\]
where the implicit constant depends on ${\|\phi\|_{\dot H^1}}$.
\item If $\mu=-1$ (the focusing case), the same result holds provided
\begin{equation}\label{embed-subthreshold}
 E_0[T_0^{-1}\phi]< E_{ 0}[W_{0}] \qtq{and} \|\Re \phi\|_{\dot H_x^1}< \|W_{0}\|_{\dot H_{0}^1},
\end{equation}
if $t_n\equiv 0$, and
\begin{equation}\label{embed-subthreshold2}
\tfrac12\|\phi\|_{\dot H_x^1}^2 < E_0[W_0]\qtq{and} \|\Re \phi\|_{\dot H_x^1}< \|W_{0}\|_{\dot H_{0}^1},
\end{equation}
if $t_n\to\pm\infty$.
\item Furthermore, for every $\eta>0$ there exists $N_\eta$ and $\psi_\eta\in\C_c^\infty(\R^{1+d})$ such that for $n\geq N_\eta$,
\[
\|v_n(t-\lambda_n t_n,x+x_n)-\lambda_n^{-[\frac{d}{2}-1]}\psi_\eta(\lambda_n^{-1} t,\lambda_n^{-1}x)\|_{L_{t,x}^{\frac{2(d+1)}{d-2}}(\R^{1+d})} < \eta.
\]
\end{itemize}
\end{proposition}

\begin{proof}  Our ultimate goal is to construct solutions $v_n$ to \eqref{Tnlw} with data $v_n(0)=\phi_n$.  Equivalently, we need to construct solutions $u_n$ to \eqref{nlw} with data $\vec u_n(0)=T_a^{-1}\phi_n$.  The starting point is to appeal to Theorem~\ref{T:NLW} to construct a solution $u$ associated to the initial data $(\Re\phi, |\nabla|\Im\phi).$  The assumptions \eqref{embed-subthreshold} and \eqref{embed-subthreshold2} guarantee that we are in a position to apply Theorem~\ref{T:NLW}.

If $t_n\equiv 0$, we take $u$ to be the solution to \eqref{nlw0} with initial data
\[
\vec\psi:=(\Re\phi,|\nabla|\Im\phi).
\]
If $t_n\to\pm\infty$, we instead of $u$ be the solution to \eqref{nlw0} with
\begin{equation}\label{free-scatter}
\lim_{t\to\pm\infty}\| \vec u(t) - (S_0(t)\vec\psi, \partial_t S_0(t)\vec\psi)\|_{\dot H^1\times L^2}=0,
\end{equation}
where $S_0(t)(f,g)=\cos(t|\nabla|)f + |\nabla|^{-1}\sin(t|\nabla|)g$ is the free linear wave propagator.

In both cases, we have that $u$ obeys global space-time bounds.

We will now use $u$ to construct approximate solutions to \eqref{nlw}.  For each $n$, we let $\chi_n$ be a smooth function such that
\[
\chi_n(x) =\begin{cases} 0 & |x_n+\lambda_n x|\leq \tfrac14 |x_n|, \\ 1 & |x_n+\lambda_n x|\geq \tfrac12|x_n|.\end{cases}
\]
In particular, $\chi_n(x)\to 1$ as $n\to\infty$ for each $x$. We further impose that $\chi_n$ obey the symbol bounds
\[
\sup_x |\partial^\alpha \chi_n(x)| \lesssim [\lambda_n^{-1} |x_n|]^{-|\alpha|}
\]
for all multi-indices $\alpha$.

For $\tau>0$, we now let
\[
u_{n,\tau}(t,x)=\begin{cases} \lambda_n^{-[\frac{d}{2}-1]}(\chi_n u)(\lambda_n^{-1}t,\lambda_n^{-1}(x-x_n)) & |t|\leq \lambda_n\tau, \\
[S(t-\tau\lambda_n)\vec u_{n,\tau}(\lambda_n\tau)](x) & t>\lambda_n \tau, \\
[S(t+\tau\lambda_n)\vec u_{n,\tau}(-\lambda_n\tau)](x) & t<-\lambda_n\tau,\end{cases}
\]
where $S(t)(f,g)=\cos(t\sqrt{\mathcal{L}_a})f+\mathcal{L}_a^{-\frac12}\sin(t\sqrt{\mathcal{L}_a})g.$
We claim that the $u_{n,\tau}$ are approximate solutions to \eqref{nlw} that asymptotically agree with $T_a^{-1}\phi_n$, so that we may appeal to the stability result (Proposition~\ref{P:stab}) to construct true solutions to \eqref{nlw} with initial data $T_a^{-1}\phi_n$.  To do this requires that we verify the following:
\begin{align}
&\limsup_{\tau\to\infty}\limsup_{n\to\infty}\bigl\{ \|\vec u_{n,\tau}\|_{L_t^\infty(\dot H^1\times L^2)} + \|u_{n,\tau}\|_{L_{t,x}^{\frac{2(d+1)}{d-2}}}\bigr\} \lesssim 1, \label{nembed1} \\
&\limsup_{\tau\to\infty}\limsup_{n\to\infty}\| \vec u_{n,\tau}(\lambda_nt_n)-T_a^{-1}\phi_n\|_{\dot H^1\times L^2} = 0, \label{nembed2} \\
&\limsup_{\tau\to\infty}\limsup_{n\to\infty}\|(\partial_t^2+\L_a)u_{n,\tau} + F(u_{n,\tau})\|_{L_t^1 L_x^2} = 0, \label{nembed3}
\end{align}
where we have denoted $F(z)=\mu |z|^{\frac{4}{d-2}}z$ and $\vec u_{n,\tau}=(u_{n,\tau},\partial_t u_{n,\tau})$.

We begin by estimating
\[
\|\vec u_{n,\tau}\|_{L_t^\infty(\dot H^1\times L^2)} \lesssim \| \chi_n \vec u\|_{\dot H^1\times L^2} \lesssim 1.
\]
The space-time bound in \eqref{nembed1} then follows from Strichartz and the corresponding bounds for $u$.

We turn to \eqref{nembed2}.  We begin with the case $t_n\equiv 0$.  By construction and a change of variables, we estimate the $\dot H^1$ component by
\[
\| (1-\chi_n)\Re\phi \|_{\dot H^1} =o(1)\qtq{as}n\to\infty.
\]
We turn to the $L^2$ component.  Again, by construction and a change of variables, we have
\[
\| \chi_n |\nabla|\Im\phi - (\L_a^n)^{\frac12}\Im \phi\|_{L^2}=o(1)\qtq{as}n\to\infty,
\]
where we have also made use of \eqref{cop5}.

We turn to the case $t_n\to\infty$, with the case $t_n\to-\infty$ being similar.  Note that $t_n>\tau$ for $n$ sufficiently large.  It is enough to prove
\begin{equation}\label{equ:46equ}
\limsup_{\tau\to\infty}\limsup_{n\to\infty}\big\|T_a\vec{u}_{n,\tau}(\lambda_nt_n)-\phi_n\big\|_{\dot{H}^1}=0.
\end{equation}
To this end, set $u(t,x):=S(t)(f,g).$ Then
\[
T_a\vec{u}(t)=e^{-it\sqrt{\mathcal{L}_a}}T_a(f,g),
\]
which implies
\[
T_a\vec{u}_{n,\tau}(\lambda_nt_n)=e^{-i(\lambda_nt_n-\lambda_n\tau)\sqrt{\mathcal{L}_a}}T_a\vec{u}_{n,\tau}(\lambda_n\tau).
\]
Thus, performing a change of variables, we have
\begin{align*}
\big\|T_a\vec{u}_{n,\tau}(\lambda_nt_n)-\phi_n\big\|_{\dot{H}^1}&=  \big\|e^{-i(\lambda_nt_n-\lambda_n\tau)\sqrt{\mathcal{L}_a}}T_a\vec{u}_{n,\tau}(\lambda_n\tau)-e^{-it_n\lambda_n\sqrt{\mathcal{L}_a}}g_n\phi\big\|_{\dot{H}^1}\\
&\lesssim\big\|T_a\vec{u}_{n,\tau}(\lambda_n\tau)-e^{-i\lambda_n\tau\sqrt{\mathcal{L}_a}}g_n\phi\big\|_{\dot{H}^1}\\
&\lesssim\big\|\chi_nT_a^n\vec{u}(\tau)-e^{-i\tau\sqrt{\mathcal{L}_a^n}}g_n\phi\big\|_{\dot{H}^1}\\
&\lesssim\big\|T_a^n\vec{u}(\tau)-e^{-i\tau\sqrt{\mathcal{L}_a^n}}\phi\big\|_{\dot{H}^1}+o_n(1)
\end{align*}
as $n\to\infty$, where $T_a^n(f,g):=f+i(\mathcal{L}_a^n)^{-\frac12}g.$ Furthermore, using \eqref{cop5}, Corollary~\ref{C:COO} and \eqref{free-scatter}, we derive that
\begin{align*}
\big\|&T_a\vec{u}_{n,\tau}(\lambda_nt_n)-\phi_n\big\|_{\dot{H}^1}\\ &\lesssim \big\|T_a^n\vec{u}(\tau)-T_0\vec{u}(\tau)\big\|_{\dot{H}^1}  +\big\|e^{-i\tau\sqrt{-\Delta}}\phi-e^{-i\tau\sqrt{\mathcal{L}_a^n}}\phi\big\|_{\dot{H}^1}\\
& \quad +\big\|T_0\vec{u}(\tau)-e^{-i\tau\sqrt{-\Delta}}\phi\big\|_{\dot{H}^1}+o_n(1)\\
 & \lesssim\big\|\partial_tu(\tau)-\sqrt{\mathcal{L}_a^n}|\nabla|^{-\frac12}\partial_tu(\tau)\big\|_{L^2}+o_n(1)\\
  &\lesssim o_n(1)
\end{align*}
as $n\to\infty.$

Turning to \eqref{nembed3}, we define the errors
\[
e_{n,\tau} = (\partial_t^2 + \L_a)u_{n,\tau} + F(u_{n,\tau}),\quad F(z)=\mu|z|^{\frac{4}{d-2}}z,
\]
which we need to estimate in $L_t^1 L_x^2(\R^{1+d})$.

We first consider the contribution of times $t>\lambda_n\tau$, with the case $t<-\lambda_n\tau$ being analogous. In this regime, we have
\[
e_{n,\tau}=F(u_{n,\tau}).
\]
Thus, by construction and a change of variables, we have
\begin{align*}
\|e_{n,\tau}\|_{L_t^1 L_x^2(\{t>\lambda_n\tau\}\times\R^d)} & \lesssim \|u_{n,\tau}\|_{L_t^{\frac{d+2}{d-2}}L_x^{\frac{2(d+2)}{d-2}}(\{t>\lambda_n\tau\}\times\R^d)}^{\frac{d+2}{d-2}} \\
& \lesssim \|S_n(t)[\chi_n \vec u(\tau)]\|_{L_t^{\frac{d+2}{d-2}}L_x^{\frac{2(d+2)}{d-2}}((0,\infty)\times\R^d)}^{\frac{d+2}{d-2}},
\end{align*}
where
\[
S_n(t)(f,g)=\cos(t\sqrt{\L_a^n})f+\tfrac{\sin(t\sqrt{\L_a^n})}{\sqrt{\L_a^n}}g.
\]
We claim that this term tends to zero as $n,\tau\to\infty$, which will yield \eqref{nembed3} in the region $|t|>\lambda_n\tau$.  Recalling the notation $\vec\psi$ from \eqref{free-scatter} and observing that we can replace $\chi_n$ with $1$ up to errors that are $o(1)$ as $n\to\infty$, we are led to estimate
\begin{align}
\|S_n(t)\vec u(\tau)\|_{S(0,\infty)} & \lesssim \|S_0(t)\vec\psi\|_{S(\tau,\infty)} \label{nembed10} \\
& \quad +\| [S_n(t)-S_0(t)]\vec{u}(\tau)\|_{S(0,\infty)} \label{nembed11}\\
& \quad + \big\| S_0(t)\big[\vec{u}(\tau)-\big(S_0(\tau)\vec{\psi},\partial_tS_0(\tau)\vec{\psi}\big)\big]\big\|_{S(0,\infty)}.\label{nembed12}
\end{align}
Now \eqref{nembed10} is $o(1)$ as $\tau\to\infty$ by Strichartz and monotone convergence.  Next, \eqref{nembed11} is $o(1)$ for each $\tau$ by Corollary~\ref{C:COO}.  Finally, \eqref{nembed12} is $o(1)$ as $\tau\to\infty$ by Strichartz and \eqref{free-scatter}.  This completes the proof of \eqref{nembed3} in the region $|t|>\lambda_n\tau$.

Finally, we turn to \eqref{nembed3} in the region $|t|\leq\lambda_n\tau$.  Recalling that $u$ is a solution to \eqref{nlw0}, we compute that in this region
\begin{align}
e_{n,\tau} & = \lambda_n^{-(\frac{d}{2}+1)}\mu[(\chi_n-\chi_n^{\frac{d+2}{d-2}})F(u)](\lambda_n^{-1} t,\lambda_n^{-1}(x-x_n)) \label{nerror1} \\
& \quad + 2\lambda_n^{-(\frac{d}{2}+1)}[\nabla\chi_n\cdot\nabla u](\lambda_n^{-1} t,\lambda_n^{-1}(x-x_n)) \label{nerror2} \\
& \quad + \lambda_n^{-(\frac{d}{2}+1)}[\Delta\chi_n u](\lambda_n^{-1} t,\lambda_n^{-1}(x-x_n)) \label{nerror3} \\
& \quad + \lambda_n^{-(\frac{d}{2}-1)}a|x|^{-2}[\chi_n u](\lambda_n^{-1} t,\lambda_n^{-1}(x-x_n)). \label{nerror4}
\end{align}

Changing variables, we estimate the contribution of \eqref{nerror2} and \eqref{nerror3} by
\begin{align*}
\tau&\bigl\{\|\nabla\chi_n\|_{L^\infty}\|\nabla u\|_{L^2} + \|\Delta\chi_n\|_{L^d}\|u\|_{L^{\frac{2d}{d-2}}}\bigr\} \lesssim\tau \tfrac{\lambda_n}{|x_n|} = o(1)
\end{align*}
as $n\to\infty$.

For \eqref{nerror1}, we change variables and observe that $F(u)\in L_t^1 L_x^2$ (since $u$ obeys $L_t^{\frac{d+2}{d-2}}L_x^{\frac{2(d+2)}{d-2}}$ bounds); thus the contribution of this term is $o(1)$ as $n\to\infty$ by the dominated convergence theorem.

Finally, for \eqref{nerror4} we will use Hardy's inequality and a change of variables.  Recalling the notation $g_n$ from above, first observe that in the support of $g_n\chi_n$, we have $|x|\gtrsim |x_n|$.  Thus
\begin{align*}
\|& |x|^{-2}g_n[\chi_n u(\lambda_n^{-1}t)]\|_{L_t^1 L_x^2(\{|t|\leq \lambda_n \tau\}\times\R^d)}\\
 & \lesssim \tfrac{\lambda_n}{|x_n|} \| |x|^{-1} g_n[\chi_n u(\lambda_n^{-1}t)]\|_{L_t^\infty L_x^2} \\
& \lesssim \tfrac{\lambda_n}{|x_n|}\|\nabla g_n(\chi_n u(\lambda_n^{-1}t))\|_{L_t^\infty L_x^2} \\
& \lesssim \tfrac{\lambda_n}{|x_n|}\| \lambda_n^{-\frac{d}{2}}\nabla [\chi_n u](\lambda_n^{-1} t, \lambda_n^{-1}(x-x_n))\|_{L_t^\infty L_x^2} \\
& \lesssim \tfrac{\lambda_n}{|x_n|}\|\nabla[\chi_n u]\|_{L_t^\infty L_x^2} = o(1)
\end{align*}
as $n\to\infty$.   This completes the proof of \eqref{nembed3}.

Applying Proposition~\ref{P:stab}, we deduce that for $n$ sufficiently large exist true solutions $\tilde u_n$ to \eqref{nlw} with initial data $T_a^{-1}\phi_n$.  Furthermore, this solution obeys global space-time bounds.  We now define $v_n=T_a\vec \tilde u_n$ to obtain the desired solutions to \eqref{Tnlw}.

Finally, the approximation result follows from the same argument in \cite{KMVZZ-NLS}.
 \end{proof}

\subsection{Reduction to almost periodic solutions}\label{S:reduction}

In this section we prove the following theorem.
\begin{theorem}\label{T:reduction} Suppose Theorem~\ref{T:defocusing} or Theorem~\ref{T:focusing} fails.  Then there exists a maximal-lifespan solution $u:I_{\max}\times\R^4\to\R$ to \eqref{nlw} that blows up in both time directions and is almost periodic modulo symmetries with $x(t)\equiv 0$.

In the focusing case, we have
\[
E_a[\vec u(0)] < E_{a\wedge 0}[W_{a\wedge 0}] \qtq{and} \|u(0)\|_{\dot H_a^1} < \|W_{a\wedge 0}\|_{\dot H_{a\wedge 0}^1}.
\]
\end{theorem}

We define
\[
L(\E) =\sup\bigl\{\|u\|_{L_{t,x}^{\frac{2(d-2)}{d+1}}(I\times\R^d)}\bigr\},
\]
where the supremum is taken over all maximal-lifespan solutions $u:I\times\R^d\to\R$ to \eqref{nlw} such that $E_a[\vec u]\leq \E$.  In the focusing case, we also restrict to solutions satisfying
\[
\|u(t)\|_{\dot H_a^1}\leq \|W_{a\wedge 0}\|_{\dot H_{a\wedge0}^1}
\]
for some $t\in I$.  By the small-data theory, we have that $L(\E)<\infty$ for $\E$ small enough.  Therefore, if Theorem~\ref{T:defocusing} or Theorem~\ref{T:focusing} fails, there exists a critical $\E_c\in(0,\infty)$ (in the defocusing case) or $\E_c\in(0,E_{a\wedge 0}[W_{a\wedge 0}])$ (in the focusing case) such that
\[
L(\E)<\infty\qtq{for}\E<\E_c\qtq{and}L(\E)=\infty\qtq{for}\E>\E_c.
\]

The key to establishing Theorem~\ref{T:reduction} is the following convergence result.

\begin{proposition}\label{P:PS} Suppose $u_n:I_n\times\R^d\to\R$ is a sequence of solutions to \eqref{nlw} with
\begin{equation}\label{EtoEc}
E_a[\vec u_n]\to \E_c,
\end{equation} and suppose $t_n\in I_n$ are such that
\begin{equation}\label{unblowup}
\lim_{n\to\infty}\|u_n\|_{L_{t,x}^{\frac{2(d+1)}{d-2}}(\{t>t_n\}\times\R^d)} = \lim_{n\to\infty} \|u_n\|_{L_{t,x}^{\frac{2(d+1)}{d-2}}(\{t<t_n\}\times\R^d)}=\infty.
\end{equation}
In the focusing case, assume additionally that
\begin{equation}\label{un-sub}
\|u_n(t_n)\|_{\dot H_a^1} \leq \|W_{a\wedge 0}\|_{\dot H_{a\wedge 0}^1}.
\end{equation}
Then, passing to a subsequence, the sequence $\{u(t_n),\partial_tu(t_n)\}$ converges in $\dot H^1\times L^2$ modulo scaling.
\end{proposition}

Assuming Proposition~\ref{P:PS}, the proof of Theorem~\ref{T:reduction} is straightforward.  If Theorem 1.1 or Theorem 1.2  fails, one can find a sequence of solutions and times satisfying the hypotheses of Proposition~\ref{P:PS}.  Therefore, one can extract (after rescaling) a subsequential limit.  The solution $v$ to \eqref{nlw} with this initial data satisfies the conclusions of Theorem~\ref{T:reduction}.  To check the compactness, for example, one applies Proposition~\ref{P:PS} with $u_n\equiv v$ for any sequence $t_n$ in the orbit of $v$.

Thus, it remains to establish Proposition~\ref{P:PS}.

\begin{proof}[Proof of Proposition~\ref{P:PS}] By time-translation symmetry, we may assume $t_n\equiv 0$.  As we developed the requisite concentration-compactness tools for the operator $e^{-it\sqrt{\L_a}}$, we will generally apply the mapping $T_a$ introduced in \eqref{T} and work with solutions to \eqref{Tnlw}.

We apply the linear profile decomposition (Proposition~\ref{P:LPD}) to the sequence
\[
T_a\vec{u}_n(0)=u_n(0)+i\L_a^{-\frac12}\partial_t u_n(0)
\]
to write
\[
T_a\vec{u}_n(0)=\sum_{j=1}^J \phi_n^j + r_n^J
\]
with all of the properties stated in Proposition~\ref{P:LPD}.   We need to prove that $J^*=1$, $r_n^1\to 0$ in $\dot H^1$, $x_n^1\equiv 0$, and $t_n^1\equiv 0$.

Note that by the decouplings \eqref{E:LP2} and \eqref{E:LP6}, we have
\[
\lim_{n\to\infty}\biggl\{ E_a[\vec u_n] - \sum_{j=1}^J \tilde E_a[\phi_n^j]-\tilde E_a[r_n^J]\biggr\}=0,
\]
where we recall the notation from \eqref{Etilde}.

Let us first show that
\[
\liminf_{n\to\infty} \tilde E_a[\phi_n^j]>0\qtq{for each}j.
\]
To see this, first observe that
\begin{equation}\label{ifxnj}
(\lambda_n^j)^{-1}|x_n^j|\to\infty \implies \|\phi_n^j\|_{\dot H_a^1}\to \|\phi^j\|_{\dot H^1}>0,
\end{equation}
which is a consequence of \eqref{cop5}.  Thus, the claim follows from \eqref{E:LP2}, \eqref{un-sub}, and Lemma~\ref{L:coercive}.  Similarly, we deduce $\liminf_{n\to\infty} \tilde E_a[r_n^J]\geq 0$ for each $J$.

There are now two possible cases.

\emph{Case 1.} Suppose $\sup_j \limsup_{n\to\infty} \tilde E_a[\phi_n^j]=\E_c$.

In this case, the energy decoupling and \eqref{EtoEc} imply that $J^*=1$, and hence we can write
\[
T_a\vec u_n(0) = \phi_n + r_n,
\]
and in fact we can deduce that $r_n\to 0$ in $\dot H^1$.  It therefore remains to preclude $\lambda_n^{-1}|x_n|\to\infty$ and $t_n\to\pm\infty$.

To this end, first suppose $\lambda_n^{-1}|x_n|\to\infty$.  We will apply Proposition~\ref{P:NP} to the profile $\phi_n$.  If $t_n\equiv 0$, then the hypotheses of Proposition~\ref{P:NP} follow from \eqref{ifxnj}, the fact that $r_n\to 0$ in $\dot H^1$, Lemma~\ref{L:COO} and Corollary~\ref{C:comparison}.  If instead $t_n\to\pm\infty$ then we utilize Corollary~\ref{C:CooL4}, as well.  Thus, by Proposition~\ref{P:NP}, for $n$ large there exists a global solution $v_n$ to \eqref{Tnlw} with $(v_n(0),\partial_t v_n(0))=\phi_n$ satisfying global space-time bounds.  Then $T_a^{-1}v_n$ is a global solution to \eqref{nlw} with global space-time bounds.  However, noting that
\[
\|(u_n(0),\partial_t u_n(0))-T_a^{-1}\phi_n\|_{\dot H^1\times L^2} \lesssim \|T_a\vec u_n(0)-\phi_n\|_{\dot H^1}\to 0,
\]
we can therefore apply the stability result (Proposition~\ref{P:stab}) to deduce that the $u_n$ obey global spacetime bounds, contradicting \eqref{unblowup}.  We conclude that $x_n\equiv 0$.

Next, if $t_n\to\infty$, we observe that by Strichartz, monotone convergence, $r_n\to 0$ in $\dot H^1$, and $x_n\equiv 0$, we have
\begin{align*}
\| e^{-it\sqrt{\L_a}}T_a\vec u_n(0)\|_{L_{t,x}^r(\{t>0\}\times\R^d)} & \lesssim \| e^{-it\sqrt{\L_a}}r_n\|_{L_{t,x}^r} + \| e^{-it\sqrt{\L_a}}\phi\|_{L_{t,x}^r((t_n,\infty)\times\R^d)} \\
& \to 0 \qtq{as}n\to\infty,
\end{align*}
where $r=\frac{2(d+1)}{d-2}$. By the small-data theory, this again implies global space-time bounds for the $u_n$, yielding a contradiction.  A similar argument precludes the possibility that $t_n\to-\infty$.

It therefore remains to preclude the following case:

\emph{Case 2.} Suppose towards a contradiction that
\[
\sup_j \limsup_{n\to\infty} \tilde E_a[\phi_n^j]<\E_c - 3\delta\qtq{for some}\delta>0.
\]

In this case, for each finite $J\leq J^*$, we have
\[
\tilde E_a[\phi_n^j] \leq \E_c-2\delta\qtq{for}1\leq j\leq J\qtq{and}n\quad\text{large}.
\]
Recalling \eqref{E:LP2}, \eqref{un-sub}, and Lemma~\ref{L:coercive}, we also have
\begin{equation}\label{profilesbelow}
\|\Re \phi_n^j\|_{\dot H^1_a} <(1-\delta') \|W_{a\wedge 0}\|_{\dot H^1_{a\wedge 0}}\qtq{for}1\leq j\leq J\qtq{and}n\quad\text{large}.
\end{equation}
We now introduce nonlinear solutions to \eqref{Tnlw} associated to each $\phi_n^j$ as follows:
\begin{itemize}
\item If $(\lambda_n^j)^{-1}|x_n^j|\to\infty$ then, arguing as above, the hypotheses of Proposition~\ref{P:NP} hold for $\phi^j$ and hence we obtain a global solution $v_n^j$ to \eqref{Tnlw} with $v_n^j(0)=\phi_n^j$.
\item If $x_n^j\equiv 0$ and $t_n^j\equiv 0$, then we let $v^j$ be the maximal-lifespan solution to \eqref{Tnlw} with $v^j(0)=\phi^j$.
\item If $x_n^j\equiv 0$ and $t_n^j\to\pm\infty$, we use Proposition~\ref{P:LWP} to find the maximal lifespan solution $v^j$ to \eqref{Tnlw} that scatters to $e^{-it\sqrt{\L_a}}\phi^j$ in $\dot H^1$ as $t\to\pm\infty$.
\end{itemize}
In the latter two cases, we define
\[
v_n^j(t,x) = (\lambda_n^j)^{-(\frac{d}{2}-1)}v^j(\tfrac{t}{\lambda_n^j}+t_n^j,\tfrac{x}{\lambda_n^j}).
\]
In particular, $v_n^j$ is also a solution to \eqref{Tnlw} with $0$ in the maximal-lifespan for large enough $n$ and satisfying
\[
\lim_{n\to\infty}\|v_n^j(0)-\phi_n^j\|_{\dot H^1} =0.
\]
In particular, it follows that $\tilde E_a[v_n^j]\leq \E_c-\delta$ for $1\leq j\leq J$ and $n$ large enough.  By the definition of $\E_c$, \eqref{profilesbelow}, and Proposition~\ref{P:NP} (for those $j$ for which $(\lambda_n^j)^{-1}|x_n^j|\to\infty$), we have that each $v_n^j$ is global in time with uniform space-time bounds; moreover, (again using Proposition~\ref{P:NP} if $(\lambda_n^j)^{-1}|x_n^j|\to\infty$) for any $\eta>0$ we may find $\psi_\eta^j\in C_c^\infty(\R^{1+d})$ such that
\[
\|v_n(t-\lambda_n t_n,x+x_n)-\lambda_n^{-(\frac{d}{2}-1)}\psi_\eta(\lambda_n^{-1} t,\lambda_n^{-1}x)\|_{L_{t,x}^{\frac{2(d+1)}{d-2}}(\R^{1+d})} <\eta
\]
for $n$ sufficiently large.

We will now construct approximate solutions to \eqref{nlw} that asymptotically match $T_a\vec u_n(0)$, but which have uniform space-time bounds.  Using the stability result (Proposition~\ref{P:stab}), this will lead to a contraction to \eqref{unblowup}.

We define
\begin{equation}\label{approx0}
w_n^J = \sum_{j=1}^J v_n^j(t)+ e^{-it\sqrt{\L_a}}r_n^J,
\end{equation}
which we immediately observe satisfies
\[
\lim_{n\to\infty} \|w_n^J(0) - T_a\vec u_n(0)\|_{\dot H^1} = 0 \qtq{for all}J.
\]
We claim that it remains to prove the following lemma.
\begin{lemma}[Approximate solutions]\label{L:approx} The functions $w_n^J$ satisfy
\begin{equation}
\limsup_{n\to\infty} \bigl\{ \|w_n^J(0)\|_{\dot H^1} + \|w_n^J\|_{S(\R)} \bigr\} \lesssim 1 \qtq{uniformly in}J, \label{E:approx1}
\end{equation}
and
\begin{equation}
\lim_{J\to J^*}\limsup_{n\to\infty} \|\L_a^{\frac12}\bigl[(i\partial_t-\L_a^{\frac12})w_n^J -\mu\L_a^{-\frac12}|\Re w_n^J|^{\frac{4}{d-2}}\Re w_n^J\bigr]\|_{L_t^1 L_x^2} = 0. \label{E:approx2}
\end{equation}
\end{lemma}

Indeed, with Lemma~\ref{L:approx} in place, we can use Proposition~\ref{P:stab} and \eqref{approx0} to deduce that the solutions $T_a\vec u_n$ to \eqref{Tnlw} inherit the uniform space-time bounds of the $u_n^J$ for large $n$, contradicting \eqref{unblowup}.

The proof of Lemma~\ref{L:approx} follows along standard lines, so we will be somewhat brief.  One essential ingredient is the orthogonality of parameters given in \eqref{E:LP5}.  In particular, \eqref{E:LP5} and approximation by functions that are $C_c^\infty$ in space-time imply the following:
\begin{lemma}[Orthogonality]\label{L:orthogonal} For $j\neq k$, we have
\[
\lim_{n\to\infty} \| v_n^j v_n^k\|_{L_{t,x}^{\frac{d+1}{d-2}}} + \|v_n^j v_n^k\|_{L_t^{\frac{d+2}{2(d-2)}}L_x^{\frac{d+2}{d-2}}} = 0.
\]
\end{lemma}

We turn to Lemma~\ref{L:approx}.
\begin{proof}[Proof of Lemma~\ref{L:approx}] The $\dot H^1$ bound in \eqref{E:approx1} is straightforward.  Using this and decoupling, we deduce
\[
\limsup_{n\to\infty}\sum_{j=1}^J \|\phi_n^j\|_{\dot H^1}^2 \lesssim 1
\]
uniformly in $J$.  Utilizing \eqref{ifxnj} for those $j$ with $(\lambda_n^j)^{-1}|x_n^j|\to\infty$, this implies
\[
\sum_{j=1}^\infty \|\phi^j\|_{\dot H^1}^2 \lesssim 1.
\]
Thus for $J_0$ sufficiently large (depending on the small-data threshold), we can use the small-data theory to deduce
\[
\sup_J \limsup_{n\to\infty} \sum_{j=J_0}^J \|v_n^j\|_{S(\R)}^2 \lesssim \sum_{j\geq J_0}\|\phi^j\|_{\dot H^1}^2 \ll 1,
\]
from which we then get
\[
\limsup_{n\to\infty}\sum_{j=1}^J \|v_n^j\|_{S(\R)}^2 \lesssim 1\qtq{uniformly in}J.
\]
Writing $r=\tfrac{2(d+1)}{d-2}$, we use Lemma~\ref{L:orthogonal} to estimate
\[
\biggl| \biggl\|\sum_{j=1}^J v_n^j \biggr\|_{L_{t,x}^r}^r - \sum_{j=1}^J \|v_n^j\|_{L_{t,x}^r}^r\biggr| \lesssim_J \sum_{j\neq k} \|v_n^j\|_{L_{t,x}^r}^{r-2} \|v_n^j v_n^k\|_{L_{t,x}^{\frac{r}{2}}} \to 0
\]
as $n\to\infty$. As the remainder term $r_n^J$ is controlled in $L_{t,x}^r$ uniformly, we deduce the $L_{t,x}^r$ bound appearing in \eqref{E:approx1}.

We turn to \eqref{E:approx2} and set
\[
F(z)=\mu\bigl[|\Re z|^{\frac{4}{d-2}}\Re z\bigr],
\]
so that (using that each $v_n^j$ solves \eqref{Tnlw})
\begin{align}
\L_a^{\frac12}(i\partial_t - \L_a^{\frac12})w_n^J - F(w_n^J) & = \sum_{j=1}^J F(v_n^j) - F(\sum_{j=1}^J v_n^j) \label{enj1} \\
& \quad + F(w_n^J - e^{-it\sqrt{\L_a}}r_n^J) - F(w_n^J). \label{enj2}
\end{align}
In particular, we need to estimate \eqref{enj1} and \eqref{enj2} in $L_t^1 L_x^2$.

First, by Lemma~\ref{L:orthogonal},
\begin{align}
\lim_{n\to\infty}\|\eqref{enj1}\|_{L_t^1 L_x^2} & \lesssim_J \lim_{n\to\infty}\sum_{j\neq k} \|v_n^j v_n^k\|_{L_t^{\frac{d+2}{2(d-2)}}L_x^{\frac{d+2}{d-2}}}\|v_n^k\|_{L_t^{\frac{d+2}{d-2}} L_x^{\frac{2(d+2)}{d-2}}}^{\frac{6-d}{d-2}} =0\label{egd3}
\end{align}
for all $J$.  Thus
\[
\lim_{J\to J^*}\limsup_{n\to\infty} \|\eqref{enj1}\|_{L_t^1 L_x^2} = 0,
\]
as desired.

We turn to \eqref{enj2}.  Employing the vanishing condition \eqref{E:LP1} (and interpolation), we find
\begin{align*}
\lim_{J\to J^*}&\lim_{n\to\infty}\|\eqref{enj2}\|_{L_t^1 L_x^2} \\
& \lesssim \lim_{J\to J^*}\lim_{n\to\infty}\| e^{-it\sqrt{\L_a}}r_n^J\|_{L_t^{\frac{d+2}{d-2}} L_x^{\frac{2(d+2)}{d-2}}}\bigl[\|w_n^J\|_{L_t^{\frac{d+2}{d-2}} L_x^{\frac{2(d+2)}{d-2}}}+\|r_n^J\|_{\dot H^1}\bigr]^{\frac{4}{d-2}} =0,
\end{align*}
as desired.

This completes the proof of Lemma~\ref{L:approx}. \end{proof}

With Lemma~\ref{L:approx} in place, we complete the proof of Proposition~\ref{P:PS} (and hence the proof of Theorem~\ref{T:reduction}. \end{proof}

\subsection{Further reductions}

In this section, we perform some further reductions to the class of solutions constructed in Theorem~\ref{T:reduction}.  We begin with the observation that the frequency scale of an almost periodic solution obeys a local constancy property, namely, $N(t)\sim N(t')$ whenever $|t-t'|\ll N(t)^{-1}$.  This is essentially a consequence of the local theory (cf. \cite[Lemma~5.18]{KVClay}, for example). Using this, we may always divide the lifespan of an almost periodic solution into characteristic subintervals $J_k$ on which $N(t)$ is equal to some constant $N_k$, with $|J_k|\sim N_k^{-1}$.

We next record a `non-triviality' condition for almost periodic solutions.  Note that while the $\dot H^1 \times L^2$-norm of $\vec u(t)$ is bounded away from zero, each component individually may spend some time near zero.  Nonetheless, by adapting the arguments of \cite[Lemma~3.4]{KVwave} one readily observes that almost periodicity implies that for any $\delta>0$, we have
\begin{equation}\label{nontrivial}
|\{t\in[t_0,t_0+\delta N(t_0)^{-1}]:\|u(t)\|_{\dot H^1}\geq \eps\}| \geq \eps N(t_0)^{-1}
\end{equation}
for some small $\eps=\eps(\delta,u)>0$ (uniformly in $t_0$).

We will proceed in a similar fashion to \cite{KenigMerle} prove the following.

\begin{theorem}\label{T:refine} Suppose there exist almost periodic solutions to \eqref{nlw} as in Theorem~\ref{T:reduction}. Then we may find an almost periodic solution $u:I_{\max}\times\R^d\to\R$ to \eqref{nlw} conforming to one of the following two scenarios.
\begin{itemize}
\item[(i)] Let $I_0=[0,\infty)$.  Then $I_{\max}\supset I_0$, $x(t)\equiv 0$, and $\inf_{t\in I_0}N(t)\geq 1$.
\item[(ii)] Let $I_0=(0,1]$.  Then $I_{\max}\supset I_0$ with $\inf I_{\max}=0$, $x(t)\equiv 0$, and $N(t) = t^{-1}$. Furthermore, for each $t\in I_0$, $(u,\partial_t u)$ is supported in $B_t(0)$.
\end{itemize}
In the focusing case, we have
\[
E_a(u,\partial_tu) < E_{a\wedge 0}(W_{a\wedge 0},0) \qtq{and} \sup_{t\in I_0}\|u(t) \|_{\dot H_a^1} < \|W_{a\wedge 0}\|_{\dot H_{a\wedge 0}^1}.
\]
We call scenario (i) the forward-global case and scenario (ii) the self-similar case.
\end{theorem}

\begin{proof} As mentioned above, we follow the arguments in \cite{KenigMerle}.  In fact, the proof is simplified by the fact that the solutions in Theorem~\ref{T:reduction} have $x(t)\equiv 0$.

Take $u:I_{\max}\times\R^d\to\R$ as in Theorem~\ref{T:reduction}.  A standard rescaling argument shows that we may assume $N(t)\geq 1$ on half of the maximal lifespan of $u$, say $[0,T_{max})$.  We then split into two cases, namely $T_{\max}=\infty$ or $T_{\max}<\infty$.

If $T_{\max}=\infty$, then we are in scenario (i).  Thus it remains to show that if $T_{\max}<\infty$, then we may extract a solution conforming to scenario (ii).

Suppose $T_{\max}<\infty$.  By time reversal and scaling, we may assume that $I_{\max}\supset I_0=(0,1]$ with $\inf I_{\max}=0$.  A standard rescaling argument relying on almost periodicity and local well-posedness shows that we must have $N(t)\gtrsim_u t^{-1}$.

We begin by showing that for each $t\in I_0$,  $(u(t),\partial_tu(t))$ is supported in $B_t(0)$.  Using the fact that $x(t)\equiv 0$ and $N(t)\to\infty$ as $t\to 0$, we first deduce
\[
\lim_{t\to 0^+}\int_{|x|>R} |\nabla u(t,x)|^2 + |\partial_tu(t,x)|^2 \,dx = 0 \qtq{for any} R>0.
\]
Using the small-data theory and finite speed of propagation, this implies
\[
\lim_{t\to 0^+} \int_{|x|\geq\frac32R + |t-s|} |\nabla u(s,x)|^2 + |\partial_su(s,x)|^2\,dx = 0 \qtq{for any}R>0,\ s\in[0,1).
\]
Now fix $\eta>0$ and $R$ so that $\tfrac32 R < \eta$ and let $s\in (0,1]$.  Choosing $t_n\to 0^+$, we have for $n$ sufficiently large that $t_n<s$ and
\[
\{|x|\geq s+\eta\} \subset \{ |x|\geq \tfrac32 R + s-t_n\}.
\]
Thus, sending $n\to\infty$, we get
\[
\int_{|x|\geq s+\eta} |\nabla u(s,x)|^2 + |\partial_su(s,x)|^2\,dx = 0.
\]
As $\eta,s$ were arbitrary, the claim follows.

We next wish to show that $N(t)\lesssim_u t^{-1}$.  Combining this with the upper bound, we will then be able to modify the compactness modulus by a uniformly bounded function and take $N(t)=t^{-1}$.

To this end, we will apply the virial identity Lemma~\ref{L:virial} with the weight $w(x)=\tfrac12 |x|^2$.  We write
\[
M(t) = \int -\partial_t u[x\cdot \nabla u + \tfrac{d}2 u]\,dx.
\]
Because of the support properties of $(u,\partial_tu)$, we do not need to truncate the weight $w$.  In fact, using H\"older's inequality and Sobolev embedding, $|M(t)|\lesssim t \to 0$ as $t\to 0^+$.

With $G(u) = \mu\tfrac{d-2}{2d}|u|^{\frac{2d}{d-2}}$ and $V(x) = \tfrac{a}{|x|^2}$, we have
\[
\tfrac 12 u G'(u) - G(u) = \tfrac{\mu}{d}|u|^{\frac{2d}{d-2}}\qtq{and}-\tfrac12 x\cdot \nabla V = V.
\]
Thus the virial identity becomes
\begin{align*}
M'(t) = \int |\L_a u|^2 + \mu |u|^{\frac{2d}{d-2}} \,dx.
\end{align*}
In particular, using Lemma~\ref{L:coercive} in the focusing case, we deduce
\[
M'(t)\gtrsim \|u(t)\|_{\dot H^1}^2.
\]
Using the fundamental theorem of calculus (cf. $M(t)\to 0$ as $t\to 0+$), breaking into characteristic subintervals, and employing \eqref{nontrivial}, this further implies $M(t)\gtrsim t$.

Now suppose toward a contradiction that there exists $t_n\to 0^+$ so that $N(t_n)t_n\to\infty$.  We will show that ${M(t_n)}=o(t_n)$, contradicting the fact that $M(t_n)\gtrsim_u t_n$.  To see this, we fix $\eta>0$ and note that $C(\eta)<N(t_n)t_n$ for $n$ large, where $C(\cdot)$ is the compactness modulus of $u$. We then write
\[
|M(t_n)| \leq \biggl| \int_{|x|\leq \frac{C(\eta)}{N(t_n)}} \partial_tu[x\cdot \nabla u + \tfrac{d}2 u] \,dx \biggr| + \biggl| \int_{\frac{C(\eta)}{N(t_n)}\leq| x|\leq t_n} \partial_tu[x\cdot \nabla u + \tfrac{d}2 u] \,dx\biggr|.
\]
By almost periodicity, H\"older's inequality, and Sobolev embedding, the second term is controlled by $\eta\cdot t_n$. The first term is controlled by $\tfrac{C(\eta)}{N(t_n)}=o(t_n)$.  As $\eta$ was arbitrary, we conclude $M(t_n)=o(t_n)$, as desired.
\end{proof}

To complete the proof of our main results, Theorem~\ref{T:defocusing} and Theorem~\ref{T:focusing}, it therefore suffices to rule out the possibility of solutions to \eqref{nlw} as in scenarios (i) and (ii) of Theorem~\ref{T:refine}.

\section{Preclusion of the forward global case}\label{S:FG}

In this section we suppose that $u$ is an almost periodic solution to \eqref{nlw} conforming to scenario (i) in Theorem~\ref{T:refine} and derive a contradiction.  In particular, we have $I_{max}\supset[0,\infty)$, $N(t)\geq 1$, and $x(t)\equiv 0$.  Moreover, in the focusing case, $u$ is below the ground state threshold.

We will apply the virial identity Lemma~\ref{L:virial} with $w(x) = R^2 \phi(\tfrac{x}{R})$, where $\phi$ is a smooth function satisfying
\[
\phi(x) = \begin{cases} \tfrac12|x|^2 & |x|\leq 1 \\ 2 & |x|>3.\end{cases}
\]
We recall that with $G(u) = \mu\tfrac{d-2}{2d}|u|^{\frac{2d}{d-2}}$ and $V(x) = \tfrac{a}{|x|^2}$, we have
\[
\tfrac 12 u G'(u) - G(u) = \tfrac{\mu}{d}|u|^{\frac{2d}{d-2}}\qtq{and}-\tfrac12 x\cdot \nabla V = V.
\]
Applying Lemma~\ref{L:virial} with $w$ as above and employing the fundamental theorem of calculus, H\"older's inequality, and Sobolev embedding, we deduce that
\begin{align*}
\int_{t_1}^{t_2}&\int_{\R^d}|\L_a u|^2+ \mu |u|^{\frac{2d}{d-2}}\,dx\,dt \\
&\lesssim \sup_{t\in[t_1,t_2]} R\|\nabla_{t,x} u\|_{L_x^2}^2 + \mathcal{O}\biggl(\int_{t_1}^{t_2}\int_{R<|x|<3R} |\nabla u|^2 + R^{-2} |u|^2+ |u|^{\frac{2d}{d-2}}\,dx\,dt\biggr) \\
& \quad + \biggl|\int_{t_1}^{t_2} \int_{|x|>R} |\nabla u|^2 + \tfrac{a}{|x|^2}|u|^2 + |u|^{\frac{2d}{d-2}} \,dx\,dt \biggr|
\end{align*}
for any $0<t_1<t_2<\infty$.

We will seek lower bounds for the left-hand side and upper bounds for the right-hand hand side that together will yield a contradiction.

We begin with the left-hand side. Using Lemma~\ref{L:coercive} in the focusing case, we firstly observe that
\[
\int_{\R^d} |\L_a u|^2 + \mu |u|^{\frac{2d}{d-2}}\,dx\gtrsim \|u(t)\|_{\dot H^1}^2.
\]
Thus, utilizing \eqref{nontrivial} and breaking into characteristic subintervals, we deduce that
\[
\int_0^T\int_{\R^d} |\L_a u|^2 + \mu |u|^{\frac{2d}{d-2}}\,dx\gtrsim T\delta,
\]
uniformly in $T$ for some small $\delta=\delta(u)>0$.

We now let $\eta>0$.  By almost periodicity and the fact that $\inf_{t\in[0,\infty)}N(t)\geq 1$, we may choose $R=R(\eta)$ large enough that
\[
\sup_{t\in[0,\infty)}\int_{|x|>R} |\nabla u(t,x)|^2 + \tfrac{a}{|x|^2}|u(t,x)|^2 + |u|^{\frac{2d}{d-2}} \,dx \,dt < \eta.
\]
Using H\"older's inequality as well, we can take $R$ possibly even larger to guarantee that
\[
\sup_{t\in[0,\infty)}\int_{R<|x|<3R} |\nabla u|^2 + R^{-2} |u|^2 + |u|^\frac{2d}{d-2} \,dx < \eta.
\]

Combining the estimates above on an interval fo the form $[0,T]$, we deduce that
\[
T\delta \lesssim_uR + T\eta
\]
for any $T>0$.  However, choosing $\eta=\eta(u,\delta)$ sufficiently small and then $T=T(\eta)$ sufficiently large, this leads to a contradiction.  We conclude that there are no solutions to \eqref{nlw} as in scenario (i) of Theorem~\ref{T:refine}.

\section{Preclusion of the self-similar case}\label{S:SS}

In this section we preclude the possibility of self-similar almost periodic solutions as in Theorem~\ref{T:refine}.  Recall that a self-similar almost periodic solution satisfies $x(t)\equiv 0$ and $N(t)=t^{-1}$.  In particular such solutions blow up at $t=0$; furthermore, at each $t>0$ they are supported in $B_t(0)$.

We recall the notation $x^\beta=(t,x)$.

\begin{proposition}[Virial/Morawetz estimate]\label{SSvirial} Suppose $u:(0,1)\times\R^d\to\R$ is a self-similar almost periodic solution to \eqref{nlw} as in Theorem~\ref{T:refine}. For any $t_2>t_1>0$, we have
\[
\int_{t_1}^{t_2} \int_{|x|<t} \bigl[\tfrac{x^\beta}{t} \partial_\beta u +\tfrac{d-2}{2}\tfrac{u}{t}\bigr]^2\,dx \,\tfrac{dt}{t} \lesssim \log(\tfrac{t_2}{t_1})^{\frac34}.
\]
\end{proposition}

\begin{proof} We write the equation in the form
\begin{equation}\label{12281}
\partial^\alpha\partial_\alpha u = Vu + G'(u),
\end{equation}
where
\[
V(x) = a|x|^{-2}\qtq{and} G(u)=\tfrac{d-2}{2d} |u|^{\frac{2d}{d-2}}.
\]

We next introduce the function
\[
\rho=\rho(t,x) = \bigl[(1+\eps^2)t^2 - |x|^2\bigr]^{-\frac12},
\]
which satisfies
\begin{equation}\label{12282}
x^\beta\partial_\beta \rho = - \rho.
\end{equation}

We now define the space-time region
\[
S=\bigcup_{t_1<t<t_2} \{(t,x):|x|<t\} \subset\R^{1+d}.
\]
Let $n_\beta$ denote the outward-pointing unit normal vector at $(t,x)\in \partial S$. We may write
\[
\partial S = \Sigma_1 \cup \Sigma_2,
\]
where
\begin{align*}
\Sigma_1&=\bigcup_{t_1<t<t_2}\{(t,x):|x|=t\},\\
 \Sigma_2 &= \{(t_1,x):|x|<t_1\}\cup\{(t_2,x):|x|<t_2\}.
\end{align*}

Note that $u\equiv0$ on $\Sigma_1$.

We have
\begin{align}\label{12291}
n_\beta=\tfrac{1}{\sqrt{2}}(-1,\tfrac{x}{t})&\qtq{and} x^\beta n_\beta = 0,\quad (t,x)\in\Sigma_1, \\
n_\beta=(\pm 1,0)&\qtq{and} x^\beta n_\beta = \pm t,\quad (t,x)\in\Sigma_2.\nonumber
\end{align}

We multiply the equation \eqref{12281} by $\rho[x^\beta\partial_\beta u + \tfrac{d-2}{2}u]$ and integrate over $S$.  This yields
\begin{align}
0 &= \int_S \rho[x^\beta\partial_\beta u+\tfrac{d-2}{2}u][\partial^\alpha \partial_\alpha u - Vu - G'(u)] \nonumber \\
& = \int_S \rho[x^\beta \partial^\alpha(\partial_\alpha u \partial_\beta u) - \tfrac12 x^\beta \partial_\beta(\partial_\alpha u \partial^\alpha u)] \label{12283} \\
& \quad + \int_S \tfrac{d-2}{2}\rho[\partial^\alpha(u\partial_\alpha u) - \partial^\alpha u \partial_\alpha u ] \label{12284} \\
& \quad - \int_S \rho[\tfrac 12 x^\beta V\partial_\beta(u^2) + \tfrac{d-2}{2} Vu^2] \label{12285} \\
& \quad - \int_S \rho[x^\beta \partial_\beta G(u) + \tfrac{d-2}{2}G'(u) u]. \label{12286}
\end{align}

Integration by parts (using \eqref{12282}, \eqref{12291}, $\partial^\alpha x^\beta = g^{\alpha\beta}$, and $\partial_\beta x^\beta = d+1$) yields
\begin{align}
\eqref{12283}+\eqref{12284} & =
\int_S -\partial^\alpha \rho \partial_\alpha u\bigl[x^\beta\partial_\beta u + \tfrac{d-2}{2}u\bigr] \label{12292} \\
& \quad + \int_{\Sigma_2} \rho  (\partial_\alpha u)( x^\beta\partial_\beta u) g^{\alpha\gamma} n_\gamma \label{12293}\\
& \quad - \int_{\Sigma_2} \tfrac12\rho (\partial_\alpha u)( \partial^\alpha u) x^\beta n_\beta \label{12294}\\
& \quad + \int_{\Sigma_2} \tfrac{d-2}{2} \rho u(\partial_\alpha u) g^{\alpha\gamma}n_\gamma. \label{12295}
\end{align}

Further integration by parts  (using $\tfrac12 x^\beta \partial_\beta V = -V$ and \eqref{12291}) yields
\begin{align*}
\eqref{12285} = -\int_{\partial S}\tfrac12 \rho x^\beta n_\beta Vu^2 & = -\int_{\Sigma_2}\tfrac12 \rho x^\beta n_\beta Vu^2, \\
\eqref{12286} = -\int_{\partial S} \rho x^\beta n_\beta G(u) & = -\int_{\Sigma_2} \rho x^\beta n_\beta G(u).
\end{align*}

Collecting the identities above now yields
\begin{align}
\int_S \partial^\alpha \rho \partial_\alpha u\bigl[x^\beta\partial_\beta u + \tfrac{d-2}{2}u\bigr] & = \int_{\Sigma_2} \rho  (\partial_\alpha u)( x^\beta\partial_\beta u) g^{\alpha\gamma} n_\gamma\label{6e2} \\
& \quad +\int_{\Sigma_2} \tfrac{d-2}{2} \rho u(\partial_\alpha u) g^{\alpha\gamma}n_\gamma \label{6e4}\\
& \quad - \int_{\Sigma_2} \tfrac12\rho(\partial_\alpha u)(\partial^\alpha u)x^\beta n_\beta \label{6e5} \\
& \quad - \int_{\Sigma_2} \tfrac12 \rho x^\beta n_\beta [V u^2+G(u)].\label{6e6}
\end{align}

To estimate \eqref{6e2}--\eqref{6e6}, we use $\rho \leq (\eps t)^{-1}$ on $\Sigma_2$ and $x^\beta n_\beta = \pm t$ on $\Sigma_2$.  Then, since $t^{-1}\leq |x|^{-1}$ on $\Sigma_2$, we have by $\dot H^1\times L^2$ bounds (and Hardy's inequality) that
\[
\eqref{6e2}+\eqref{6e4}+\eqref{6e5}+\eqref{6e6} \lesssim \eps^{-1}.
\]

We now turn to the left-hand side.  We wish to exhibit a coercive term and control the remaining error terms.

To this end, note that
\[
\partial^\alpha\rho \partial_\alpha u = \rho^3[x^\alpha \partial_\alpha u + \eps^2 t\partial_t u ],
\]
so that the left-hand side of \eqref{6e2} is given by
\begin{align}
\partial^\alpha\rho\partial_\alpha u[x^\beta \partial_\beta u + \tfrac{d-2}{2}u] & = \rho^3(x^\beta \partial_\beta u + \tfrac{d-2}{2} u)^2 \label{ss-main1}\\
& \quad + \rho^3(x^\beta \partial_\beta u + \tfrac{d-2}{2} u)(\eps^2 t \partial_t u - \tfrac{d-2}{2}u).  \label{ss-error1}
\end{align}

We further expand \eqref{ss-error1} to write
\begin{align}
\eqref{ss-error1} & = -\tfrac{d-2}{2} \rho^3 x^\beta \partial_\beta( \tfrac12 u^2 )- \tfrac{(d-2)^2}{4}\rho^3 u^2 \label{ss-error12} \\
& \quad + \eps^2\rho^3 (x^\beta \partial_\beta u+\tfrac{d-2}{2}u)t\partial_t u.
\end{align}

An integration by parts shows
\begin{align*}
\int_S \eqref{ss-error12} & = \int_S \tfrac{(d+1)(d-2)}{4} \rho^3 u^2 + \tfrac{3(d-2)}{4}\rho^2 (x^\beta \partial_\beta \rho) u^2 - \tfrac{(d-2)^2}{4}\rho^3 u^2 \\
& \quad - \int_{\Sigma_2} \rho^3\tfrac{d-2}{2}x^\beta n_\beta \tfrac12 u^2.
\end{align*}
The first term on the right-hand side is zero, and hence using
\[
\rho^3\lesssim \eps^{-3}t^{-3} \lesssim \eps^{-3}t^{-1}|x|^{-2}\qtq{and} |x^\beta n_\beta|=t\qtq{on} \Sigma_2,
\]
 we have by Hardy's inequality
\[
\biggl| \int_S \eqref{ss-error12}\biggr| \lesssim \eps^{-3}\int_{\Sigma_2} \tfrac{u^2}{|x|^2} \lesssim \eps^{-3}.
\]

Collecting our estimates, we have so far established
\[
\int_S \rho^3(x^\beta \partial_\beta u + \tfrac{d-2}{2} u)^2 + \eps^2\rho^3 (x^\beta \partial_\beta u+\tfrac{d-2}{2}u)t\partial_t u \lesssim \eps^{-1}+\eps^{-3}.
\]

We next use
\[
\eps^2\rho^3(x^\beta\partial_\beta u + \tfrac{d-2}{2}u)t\partial_t u \leq \tfrac12\rho^3(x^\beta\partial_\beta u+\tfrac{d-2}{2}u)^2 + \eps^4 \rho^3 \tfrac12 (t\partial_t u)^2,
\]
along with the fact that
\[
\eps^4 \int_{t_1}^{t_2} \int_{|x|<t} \rho^3 t^2 (\partial_t u)^2 \,dx\,dt \lesssim \eps\log(\tfrac{t_2}{t_1})
\]
(cf. $\rho \leq \eps^{-1}t^{-1}$) to deduce
\[
\int_S \rho^3(x^\beta \partial_\beta u + \tfrac{d-2}{2} u)^2 \lesssim \eps^{-1}+\eps^{-3}+\eps\log(\tfrac{t_2}{t_1}).
\]

Noting that that $\rho^3\gtrsim t^{-3}$ for $(t,x)\in S$, we finally conclude
\[
\int_{t_1}^{t_2} \int_{|x|<t} \bigl[\tfrac{x^\beta}{t} \partial_\beta u +\tfrac{d-2}{2}\tfrac{u}{t}\bigr]^2\,dx \,\tfrac{dt}{t} \lesssim \eps^{-1}+\eps^{-3}+\eps\log(\tfrac{t_2}{t_1}).
\]
Optimizing in $\eps$ yields
\[
\int_{t_1}^{t_2} \int_{|x|<t} \bigl[\tfrac{x^\beta}{t} \partial_\beta u +\tfrac{d-2}{2}\tfrac{u}{t}\bigr]^2\,dx \,\tfrac{dt}{t} \lesssim \log(\tfrac{t_2}{t_1})^{\frac34},
\]
which completes the proof. \end{proof}

Using Proposition~\ref{SSvirial}, we will now extract a nontrivial solution to a (degenerate) elliptic equation satisfying some integrability properties.  Below we will use a unique continuation result to conclude that such a solution cannot exist, thereby reaching a contradiction to the existence of self-similar almost periodic solutions to \eqref{nlw}, as desired.

We let $B$ denote the unit ball centered at the origin.

\begin{proposition}\label{P:degenerate} Suppose there exists a self-similar almost periodic solution to \eqref{nlw} as in Theorem~\ref{T:refine}.  Then there exists a nonzero $H^1$ solution $f:B\to\R$ to
\begin{equation}\label{degenerate}
\Delta f - x\cdot \nabla^2 f x - dx\cdot \nabla f = \tfrac{d(d-2)}{4}f + a|x|^{-2}f +\mu |f|^{\frac{4}{d-2}} f
\end{equation}
satisfying $f|_{\partial B}=0$ and
\begin{equation}\label{vanishing1}
\int_B \frac{|f|^{\frac{2d}{d-2}}}{(1-|x|)^{\frac12}}\,dx + \int_{B} \frac{|\slashed{\nabla} f|^2}{(1-|x|)^{\frac12}} \,dx \lesssim 1,
\end{equation}
where $\slashed{\nabla}$ denotes the angular derivative.
\end{proposition}

\begin{proof} Suppose $u$ is a self-similar almost periodic solution.  Let us first extract the stationary solution $f$.

We first claim that Proposition~\ref{SSvirial} yields a sequence $t_n\downarrow 0$ such that
\begin{equation}\label{SSvirial2}
\int_{t_n}^{2t_n}\int_{|x|<t} \bigl[\tfrac{x^\beta}{t}\partial_\beta u + \tfrac{d-2}{2}\tfrac{u}{t}\bigr]^2\,dx\tfrac{dt}{t} \to 0.
\end{equation}
To see this, we argue as in \cite[Lemma~4.1]{DJKM}.  We apply Proposition~\ref{SSvirial} with $t_2=2^{-J}$ and $t_1=4^{-J}$ for some large $J>0$. Then Proposition~\ref{SSvirial} implies that there exists $j=j(J)$  such that
\[
\int_{2^j 4^{-J}}^{2^{j+1}4^{-J}} \int_{|x|<t} \bigl[\tfrac{x^\beta}{t}\partial_\beta u + \tfrac{d-2}{2}\tfrac{u}{t}\bigr]^2\,dx\tfrac{dt}{t} \lesssim J^{-\frac14}.
\]
Now choose a sequence $J_n\to\infty$ such that $J_n\geq 2J_{n-1}$ and take $t_n = 2^{j(J_n)}4^{-J_n}$.  Then $t_n\downarrow 0$ and \eqref{SSvirial2} holds along this sequence.

By almost periodicity (and the fact that $N(t)=t^{-1}$), we have (passing to a further subsequence)
\[
 (\tilde u_n(0), \partial_t \tilde u_n(0)):= (t_n^{\frac{d}{2}-1} u(t_n,t_n\cdot),t_n^{\frac{d}{2}}(\partial_t u)(t_n,t_n \cdot)) \to (v_0,v_1)
\]
strongly in $\dot H^1\times L^2$ for some $(v_0,v_1)\in \dot H^1\times L^2$.  Note that $(v_0,v_1)$ are supported in the unit ball.   We let $v:[0,\delta)\times\R^d\to\R$ be the solution to \eqref{nlw} with initial data $(v_0,v_1)$.

Now observe that by scaling symmetry,
\[
\tilde u_n(t,x)=t_n^{\frac{d}{2}-1}u(t_n(1+t),t_nx)
\]
is the solution to \eqref{nlw} with initial data $(\tilde u_n(0),\partial_t \tilde u_n(0))$. By a change of variables, \eqref{SSvirial2} implies
\[
\lim_{n\to\infty}\int_0^1 \int_{|y|<1}\bigl|\partial_t \tilde u_n(s,y) + \tfrac{y}{s+1}\cdot\nabla \tilde u_n(s,y)+\tfrac{d-2}{2}\tfrac{\tilde u_n(s,y)}{s+1}\bigr|^2 \,dx \,dt = 0,
\]
from which we deduce
\[
\partial_t v + \tfrac{x}{t+1}\nabla v + \tfrac{d-2}{2} \tfrac{v}{t+1}\equiv 0
\]
in $[0,\delta]\times\{|x|<1\}$. By the method of characteristics, this implies that
\[
v(t,x)=(t+1)^{-[\frac{d}{2}-1]} f(\tfrac{x}{1+t})
\]
for some $f\in H^1$ supported in the unit ball.  Combining this form with the fact that $v$ solves \eqref{nlw}, we immediately deduce that $f$ solves \eqref{degenerate}.

We turn to establishing \eqref{vanishing1}.  To begin, we collect a few properties about the solution $f$.

First, because  the solution $v$ belongs to $L_{t,x}^{\frac{2(d+1)}{d-2}}$ locally in time, a change of variables yields
\[
f\in L_x^{\frac{2(d+1)}{d-2}}.
\]

Next, we observe that
\begin{equation}\label{weightedL2}
\int \frac{|f|^2}{(1-|x|)^s}\,dx \lesssim 1 \qtq{for all} s\leq 2.
\end{equation}
The case $s=0$ is clear, while the case $s=2$ can be deduced as a consequence of Hardy's inequality.

Using \eqref{weightedL2} and H\"older's inequality, we also observe that
\begin{equation}\label{vanishing2}
\int \frac{|f|^{\frac{2d}{d-2}}}{(1-|x|)^{\frac12}}\,dx \lesssim \biggl(\int\frac{|f|^2}{(1-|x|)^{2}} \biggr)^{\frac14} \biggl(\int |f|^{\frac{2(3d+2)}{3(d-2)}}\,dx \biggr)^{\frac34}\lesssim 1,
\end{equation}
where we use
\[
\tfrac{2(3d+2)}{3(d-2)}<\tfrac{2(d+1)}{d-2}.
\]
This gives the first bound in \eqref{vanishing1}.

We turn to the second estimate in \eqref{vanishing1}.  Let us define the weight $w(x)=(1-|x|^2)^{-\frac12}$.  We multiply both sides of \eqref{degenerate} by $fw$ and integrate by parts.  This yields
\begin{align*}
-\int[|\nabla f|^2&-(x\cdot\nabla f)^2]w + \int (x\cdot \nabla f)fw +f\{(x\cdot\nabla f)(x\cdot\nabla w)-\nabla f\cdot\nabla w\}\,dx  \\
& = \int \tfrac{d(d-2)}{4}|f|^2 w + a|x|^{-2}|f|^2 w + \mu |f|^{\frac{2d}{d-2}}w\,dx.
\end{align*}
Noting that $\nabla w = xw^3$, we find that the second  term on the left-hand side vanishes.  Using \eqref{weightedL2} and \eqref{vanishing2}, we deduce that
\[
\int \frac{\bigl| |\nabla f|^2 - (x\cdot \nabla f)^2\bigr|}{(1-|x|)^{\frac12}}\,dx\lesssim 1.
\]
As
\[
|\nabla f|^2 - (x\cdot\nabla f)^2 = (1-|x|^2)|\nabla f|^2+|x|^2|\slashed{\nabla}f|^2,
\]
we deduce that \eqref{vanishing1} holds.  \end{proof}

To rule out the self-similar scenario of Theorem~\ref{T:refine}, it therefore suffices to preclude the possibility of a solution to \eqref{degenerate} as in Proposition~\ref{P:degenerate}.  For this we will rely on unique continuation results for elliptic PDE.  In fact, at this point we are in almost an identical situation to \cite[Proposition~6.12]{KenigMerle}.  Indeed the remaining issues to address are all related to the degeneracy of \eqref{degenerate} at $|x|=1$; in particular, the presence of the potential term $a|x|^{-2}f$ plays essentially no role.  For the sake of completeness, however, let us not simply quote \cite[Proposition~6.12]{KenigMerle} and conclude the proof.  Instead let us briefly go through the argument (parallel to that in \cite{KenigMerle}) to preclude the existence of a solution as in Proposition~\ref{P:degenerate}.

It remains to prove the following:
\begin{proposition}\label{P:degenerate2} Suppose $f$ is a solution to \eqref{degenerate} as in Proposition~\ref{P:degenerate}.  Then $f\equiv 0$.
\end{proposition}

\begin{proof} As just mentioned, the PDE \eqref{degenerate} is a degenerate elliptic PDE, with the degeneracy occurring as $|x|\to 1$.  In particular, by standard unique continuation results (see \cite{AKS}), the result will follow if we can prove $f\equiv 0$ on $\{1-\delta<|x|<1\}$ for some small $\delta>0$.

The idea is to introduce a change of variables that removes this degeneracy and to study the resulting PDE.  We begin by changing to polar coordinates $f=f(r,\omega)$ and rewriting the left-hand side of the PDE \eqref{degenerate} as
\[
(1-r^2)\partial_r^2 f + (\tfrac{d-1}{r}-dr)\partial_r f + \tfrac{1}{r^2}\slashed{\Delta} f,
\]
where $\slashed{\Delta}$ denotes the spherical Laplacian.  We now introduce $g(s,\omega)=f(r(s),\omega)$ for a function $r(s)$ to be defined shortly.  The left-hand side of \eqref{degenerate} becomes
\[
\tfrac{1-r^2}{(r')^2}\partial_s^2g + \bigl\{-\tfrac{(1-r^2)r''}{(r')^3}+\tfrac{d-1}{rr'}-\tfrac{dr}{r'}\bigr\}\partial_s g - \tfrac{1}{r^2}\slashed{\Delta} g,
\]
where $r=r(s)$.  If we choose
\[
r(s) = 1-\tfrac14(1-s)^2
\]
(as in \cite{KenigMerle}), then the expression above becomes
\[
(1+r)\partial_s^2 g + (1-r)^{-\frac12}\bigl\{\tfrac{d-1}{r}-(d-\tfrac12)r+\tfrac12\}\partial_s g-\tfrac{1}{r^2}\slashed{\Delta}g.
\]
Furthermore, the domain $\{1-\delta<r<1\}$ corresponds to $\Omega:=\{1-\delta'<s<1\}$, where $\delta'=2\sqrt{\delta}$.  In particular, the PDE for $g$, namely
\begin{equation}\label{nondegenerate}
(1+r)\partial_s^2  g +(1-r)^{-\frac12}\{\tfrac{d-1}{2}-(d-\tfrac12)r+\tfrac12\}\partial_s g - \tfrac{1}{r^2}\slashed{\Delta} g = \N(g),
\end{equation}
where
\[
\N(g)=\tfrac{d(d-2)}{4}g+ar^{-2}g+\mu|g|^{\frac{4}{d-2}}g,
\]
is nondegenerate and hence amenable to unique continuation results.

From this point on, the strategy is as follows:
\begin{itemize}
\item[(i)] Collect bounds on $g$ that show, in particular, that it is a standard solution to \eqref{nondegenerate}.
\item[(ii)] Define $\tilde g(s,\omega)=\chi(s)g(s,\omega)$, where $\chi$ is the characteristic function of $(0,1)$, and show that $\tilde g$ is a weak solution to \eqref{nondegenerate} on $\{1-\delta'<s<2\}$.
\end{itemize}
With (i) and (ii) in place, we can (as in \cite{KenigMerle}) invoke unique continuation (cf. \cite{AKS}) to deduce that $g\equiv 0$ on $\Omega$, and hence complete the proof.

(i) First, a change of variables in the \eqref{vanishing1} yields
\[
\int_{\Omega} |g|^{\frac{2d}{d-2}} + |\slashed{\nabla}g|^2 \lesssim 1.
\]
Similarly,
\[
\int_{\Omega} \frac{|\partial_s g|^2}{1-s}\lesssim \int |\partial_r f|^2 \lesssim 1\qtq{and} \int_{\Omega} \frac{|g|^2}{(1-s)^3}\lesssim \int \frac{|f|^2}{(1-r)^2} \lesssim 1,
\]
where we use \eqref{weightedL2} for the last bound.

(ii) We turn to (ii) define $\tilde g$ as above.  In order to write down the weak formulation of \eqref{nondegenerate}, it is useful to rewrite \eqref{nondegenerate} as
\[
\partial_s(1+r)^{\frac12}\partial_s g +\tfrac{d-1}{r}(1-r)^{\frac12}\partial_sg + (1+r)^{-\frac12}r^{-2}\slashed{\Delta}g=(1+r)^{-\frac12}\N(g).
\]
Now recall that $g$ solves \eqref{degenerate} on $\Omega$.  Thus (letting $\phi$ be a test function and integrating by parts), we find that to prove that $\tilde g$ is a weak solution reduces to proving
\[
\int \partial_s\phi(1+r)^{\frac12}\partial_s(\chi g)\,ds\,d\omega = \int\partial_s g(1+r)^{\frac12}\partial_s(\chi\phi)\,ds\,d\omega
\]
and
\[
\int g\chi\partial_s\{\tfrac{d-1}{r}(1-r)^{\frac12}\phi\}\,ds\,d\omega = -\int \{\tfrac{d-1}{r}(1-r^2)^{\frac12}\partial_s g\}\chi\phi\,ds\,d\omega.
\]
Letting $\chi_\eps$ be smooth approximations to $\chi$, the problem therefore reduces to proving
\begin{align}
&\lim_{\eps\to 0} \int \{|g\partial_s\phi|+|\partial_sg \phi|\}(1+r)^{\frac12}\partial_s\chi_\eps\,ds\,d\omega = 0, \label{uc-vanish1}\\
&\lim_{\eps\to 0}\int g\{\tfrac{d-1}{r}(1-r)^{\frac12}\phi\}\partial_s\chi_\eps\,ds\,d\omega = 0.\label{uc-vanish2}
\end{align}
The bounds established in (i) are well-suited for proving \eqref{uc-vanish1} and \eqref{uc-vanish2}.  Consider for example, the second term in \eqref{uc-vanish1}.  Assuming $\partial_s\chi_\eps$ is of size $\eps^{-1}$ supported in an interval $I_\eps=(1-2\eps,1-\eps)$, we get the bound
\[
\eps^{-1}\int_{s\in I_\eps} |\phi|\,|\partial_s g|\,ds\,d\omega \lesssim \int_{s\in I_\eps}\tfrac{|\partial_s g|}{|1-s|}\,ds \lesssim \biggl(\int_{s\in I_\eps}\frac{|\partial_s g|^2}{|1-s|}\biggr)^{\frac12}[\log(\tfrac{1-\eps}{1-2\eps})]^{\frac12},
\]
which tends to zero as $\eps\to 0$.  As the other terms can be treated similarly, this completes the proof.
\end{proof}

\section{Proof of Theorem~\ref{T:blowup}}\label{S:blowup}

In this section we give the proof of Theorem~\ref{T:blowup}, which contains two statements: (i) a blowup result below the ground state energy, and (ii) the failure of uniform space-time bounds as one approaches the ground state threshold in the case $a>0$.  The proof of (i) is similar to the blowup result appearing in the work of \cite{KenigMerle}, while the proof of (ii) follows a similar strategy as \cite{KMVZ, KVZ}.  Consequently, our presentation will be rather brief.

\begin{proof}[Proof of Theorem~\ref{T:blowup} (i)] Suppose $(u_0,u_1)\in\dot H^1\times L^2$ satisfies
\[
E_a[(u_0,u_1)]<E_{a\wedge 0}[W_{a\wedge 0}]\qtq{and}\|u_0\|_{\dot H_a^1} > \|W_{a\wedge 0}\|_{\dot H_{a\wedge 0}^1}
\]
and $u$ is the corresponding solution to \eqref{nlw}.  We will show that $u$ blows up in finite time.  To this end we introduce the function
\[
y_R(t) := \int |u(t,x)|^2 \varphi(\tfrac{x}{R})\,dx,
\]
where $R>0$ and $\varphi$ is a smooth function satisfying $\varphi\equiv 1$ for $|x|\leq 1$, $\varphi\equiv 0$ for $|x|>2$ and $0\leq \varphi\leq 1$ for $1<|x|<2$.  Direct computation using \eqref{nlw} yields 
\[
y_R'(t) = 2\int u\partial_t u \varphi(\tfrac{x}{R})\,dx
\]
and
\begin{equation}\label{y''}
y_R''(t) = 2\int |\partial_t u|^2 - |\nabla u|^2 - \tfrac{a}{|x|^2}|u|^2 + |u|^{\frac{2d}{d-2}}\,dx + r(R),
\end{equation}
where
\begin{align*}
r(R)& :=2\int [1-\varphi(\tfrac{x}{R})]\cdot[|\partial_t u|^2-|\nabla u|^2 - \tfrac{a}{|x|^2}|u|^2 + |u|^{\frac{2d}{d-2}}] \,dx \\
& \quad - \tfrac{2}{R}\int u\nabla u \cdot[\nabla \varphi](\tfrac{x}{R})\,dx. 
\end{align*}
Now, using $E_{a}[\vec u] < E_{a\wedge 0}[W_{a\wedge 0}]$, we can show that
\begin{equation}\label{equ:potlarg}
2\int|u|^{\frac{2d}{d-2}}\,dx \geq \tfrac{2d}{d-2}\int\bigl[|\partial_t u|^2 + |\nabla u|^2+ \tfrac{a}{|x|^2}|u|^2\bigr]\,dx - \tfrac{4}{d-2}\|W_{a\wedge 0}\|_{\dot H_{a\wedge 0}^1}^2 + \delta
\end{equation}
for some $\delta>0$. Combining this with \eqref{y''}, we get
\begin{align*}
y_R''(t)& \geq \tfrac{4(d-1)}{d-2}\int |\partial_t u|^2\,dx + \tfrac{4}{d-2}\bigl(\|u\|_{\dot H_a^1}^2 - \|W_{a\wedge 0}\|_{\dot H_{a\wedge 0}^1}^2\bigr) + \delta + r(R) \\
& \geq \tfrac{4(d-1)}{d-2}\int  |\partial_t u|^2 \,dx + \delta + r(R). 
\end{align*}

Now observe that if we additionally assume $u_0\in L^2$, then the formulas above make sense even as $R\to\infty$ (in which case $r(R)$ becomes identically zero). In this case an application of Cauchy--Schwarz leads to the lower bound
\begin{equation}\label{ODElb}
y(t)y''(t)\geq \tfrac{d-1}{d-2}[y'(t)]^2,
\end{equation}
from which an ODE argument yields finite time blowup.  

In the general case, we need to estimate the term $r(R)$.  To this end note that by finite speed of propagation, for any $\eps>0$ we may find $M=M(\eps)$ such that 
\[
\int_{|x|>M+t} |\partial_t u|^2 + |\nabla u|^2 + \tfrac{|a|}{|x|^2}|u|^2\,dx <\eps. 
\]
Choosing $\eps\ll\tfrac12\delta$ and $R>2M$, we deduce
\[
|r(R)| \ll \tfrac12\delta \qtq{uniformly for}t\in(0,\tfrac12R). 
\]
Thus we have
\[
y_R''(t)\geq \tfrac{4(d-1)}{d-2}\int |\partial_t u|^2\varphi(\tfrac{x}{R})\,dx\qtq{for}t\in(0,\tfrac12R),
\]
which in particular yields an estimate like \eqref{ODElb} for $y_R$ on the interval $(0,\tfrac12R)$.  In particular, a similar ODE type argument (with a careful choice of parameters) once again yields finite-time blowup.  As the complete details appear in \cite{KenigMerle} (cf. Theorem~3.7 and Theorem~7.1(ii) therein) and apply equally well in our case, we omit the details here.
\end{proof}

Finally, we turn to the proof of Theorem~\ref{T:blowup} (ii).

\begin{proof}[Proof of Theorem~\ref{T:blowup} (ii)]  Recall we are in the setting of $a>0$ and $\mu=-1$ (the focusing case).  We define
\[
\phi_n(x):=(1-\eps_n)W_0(x-x_n),
\]
where $\eps_n\to0$ and $|x_n|\to\infty.$  We can show that
\[
E_a[\phi_n]\nearrow E_0[W_0],\qtq{and} \|\phi_n\|_{\dot{H}^1_a}\nearrow \|W_0\|_{\dot{H}^1},
\]
and hence (by Theorem~\ref{T:focusing}) there exist global scattering solutions $u_n$ to \eqref{nlw} with data $(\phi_n,0)$.  Our goal is to show that the Strichartz norm of these solutions diverges as $n\to\infty$. 

To this end, we define
\[
\tilde{u}_n(t,x)=(1-\eps_n)[\chi_nW_0](x-x_n),
\]
where $\chi_n$ is as in Proposition~\ref{P:NP}, that is, a smooth function satisfying
\begin{equation*}
  \chi_n(x)=\begin{cases}
              0, & \mbox{if } |x+x_n|\leq\frac14|x_n| \\
              1, & \mbox{if }   |x+x_n|\geq\frac12|x_n|
            \end{cases}
            \quad\text{with}\quad \sup_x\big|\partial^\alpha\chi_n(x)|\lesssim|x_n|^{-|\alpha|}
\end{equation*}
for all multi-indices $\alpha.$ One can verify that
\[\|\tilde{u}_n(0)-u_n(0)\|_{\dot{H}^1}=\big\|\big[(1-\epsilon_n)\chi_n-1\big]W_0\big\|_{\dot{H}^1}\to0
\]
as $n\to\infty$, with
\[
\|\tilde{u}_n\|_{L_{t,x}^\frac{2(d+1)}{d-2}([-T,T]\times\R^d)}\gtrsim T^\frac{d-2}{2(d+1)}\qtq{for} T>0.
\]

Using the equation $-\Delta W_0=|W_0|^\frac{4}{d-2}W_0,$ we find
\begin{align}\nonumber
e_n&:=(\partial_t^2+\mathcal{L}_a)\tilde{u}_n-|\tilde{u}_n|^\frac{4}{d-2}\tilde{u}_n\\ 
  &= \big[(1-\epsilon_n)\chi_n(x-x_n)-(1-\epsilon_n)^\frac{d+2}{d-2} \chi_n(x-x_n)^\frac{d+2}{d-2}\big]|W_0|^\frac{4}{d-2}W_0 \label{equ:erren1} \\
&\quad +(1-\epsilon)\big[W_0\Delta\chi_n+2\nabla\chi_n\cdot\nabla W_0\big](x-x_n)\label{equ:erren2}\\
&\quad-\tfrac{a}{|x|^2}(1-\epsilon_n)[\chi_nW_0](x-x_n).\label{equ:erren3}
\end{align}
We now claim that for any fixed $T>0$,
\[
\|e_n\|_{L_t^1L_x^2([-T,T]\times\R^d)}\to0\qtq{as} n\to\infty.
\]

We begin with the estimate of \eqref{equ:erren1}.   As $W_0\in L^p$ for any $p>\tfrac{d}{d-2}$, we have
\begin{align*}
 \|&\eqref{equ:erren1}\|_{L_t^1L_x^2([-T,T]\times\R^d)}\\
 & \lesssim T\big\|\big[(1-\epsilon_n)\chi_n(x-x_n)-(1-\epsilon_n)^\frac{d+2}{d-2} \chi_n(x-x_n)^\frac{d+2}{d-2}\big]|W_0|^\frac{4}{d-2}W_0\big\|_{L_t^\infty L_x^2} \\
& \lesssim T\big\|\big[(1-\epsilon_n)\chi_n(x-x_n)-(1-\epsilon_n)^\frac{d+2}{d-2}
  \chi_n(x-x_n)^\frac{d+2}{d-2}\big]W_0\big\|_{L_t^\infty L_x^{2d}} \\
& \quad \times  \|W_0\|_{L_x^\frac{8d}{(d-1)(d-2)}}^\frac{4}{d-2}\\
 &\to0\quad\text{as}\quad n\to\infty.
\end{align*}
Next, we estimate of \eqref{equ:erren2}. We have
\begin{align*}
\|\eqref{equ:erren2}\|_{L_t^1L_x^2([-T,T]\times\R^d)}&\lesssim 
 T \bigl(\|\Delta\chi_n\|_{L_x^{\frac{2d}{4-d}-}}\|W_0\|_{L_x^{\frac{d}{d-2}+}}+\|\nabla \chi_n\|_{L_x^\infty}\|W_0\|_{\dot{H}^1}\bigr)\\
 &\lesssim T(|x_n|^{-\frac{d}{2}+}+|x_n|^{-1})\to0\quad\text{as}\quad n\to\infty.
\end{align*}

Finally, we have 
\begin{align*}
  \|\eqref{equ:erren3}\|_{L_t^1L_x^2([-T,T]\times\R^d)}&\lesssim 
 T \big\|\tfrac{\chi_n}{|\cdot+x_n|^2}\big\|_{L_x^{\frac{2d}{4-d}-}}\|W_0\|_{L_x^{\frac{d}{d-2}+}}\\
 & \lesssim T|x_n|^{-\frac{d}{2}+}\to0\quad\text{as}\quad n\to\infty.
\end{align*}

Applying the stability result (Proposition~\ref{P:stab}), we deduce 
\[
\|u_n\|_{L_{t,x}^\frac{2(d+1)}{d-2}([-T,T]\times\R^d)}\gtrsim T.
\]
As $T>0$ was arbitrary, this implies the result.
\end{proof}

%

\end{document}